\documentclass[final]{siamltex}
\usepackage{amssymb}
\usepackage{amsmath}
\usepackage{amsfonts}
\usepackage{algorithm}
\usepackage{graphicx}
\usepackage{color}
\usepackage{epsfig}
\usepackage{subfigure}
\usepackage{multirow}
\usepackage{mathrsfs}
\date{today}

\numberwithin{equation}{section}

\newtheorem{Theorem}{Theorem}[section]
\newtheorem{Proposition}[Theorem]{Proposition}
\newtheorem{Corollary}[Theorem]{Corollary}
\newtheorem{Lemma}[Theorem]{Lemma}

\newtheorem{Remark}[Theorem]{Remark}

\newtheorem{Definition}{Definition}[section]
\newtheorem{Example}{Example}[section]

\begin{document}

\title{Continuation Methods for Computing  Z-/H-eigenpairs of Nonnegative Tensors}
\date{}
\author{ Yueh-Cheng Kuo\thanks{%
Department of Applied Mathematics, National University of Kaohsiung,
Kaohsiung 811, Taiwan; \texttt{yckuo@nuk.edu.tw}} \and Wen-Wei Lin\thanks{%
Department of Applied Mathematics, National Chiao Tung University, Hsinchu
300, Taiwan \texttt{wwlin@math.nctu.edu.tw}} \and Ching-Sung Liu\thanks{%
Department of Applied Mathematics, National University of Kaohsiung,
Kaohsiung 811, Taiwan; \texttt{chingsungliu@nuk.edu.tw}}}
\maketitle

\begin{abstract}
In this paper, a homotopy continuation method for the computation of nonnegative
Z-/H-eigenpairs of a nonnegative tensor is presented. We show that the homotopy
continuation method is guaranteed to compute a nonnegative eigenpair.
Additionally, using degree analysis, we show that the number of positive
Z-eigenpairs of an irreducible nonnegative tensor is odd. A novel homotopy
continuation method is proposed to compute an odd number of positive Z-eigenpairs, and some numerical results are presented.
\end{abstract}




\pagestyle{myheadings} \thispagestyle{plain}
\markboth{Yueh-Cheng Kuo, Wen-Wei Lin, And  Ching-Sung Liu}{Continuation Methods for Computing  Z-/H-eigenpairs}


\section{Introduction}

An $m$th-order tensor $\mathcal{A}\in \mathbb{F}^{n_{1}\times n_{2}\times
\cdots \times n_{m}}$ is a multidimensional or $m$-way array, where $\mathbb{%
F}$ is a field. A first-order tensor is a vector, a second-order tensor is a
matrix, and tensors of order three or higher are called higher-order
tensors. When $n:=n_{1}= n_{2}=\cdots =  n_{m}$, $\mathcal{A}$ is called an $%
m$th-order, $n$-dimensional tensor. We denote the set of all $m$th-order, $n$%
-dimensional tensors on the field $\mathbb{F}$ by $\mathbb{F}^{[m,n]}$. For
a tensor $\mathcal{A}\in \mathbb{F}^{[m,n]}$, the tensor eigenvalues and
eigenvectors have been considered in many literatures \cite%
{Chang-Qi-Zhang:2013,Lim:2005,Qi:2005,Qi:2007,Qi-Sun-Wang:2007}, there are two particularly
interesting definitions called Z-eigenvalues and H-eigenvalues (see the definition later on).
Tensor eigenproblems have found applications in automatic control \cite%
{Anderson-Bose-Jury:1975,Bose-Kamt:1974,Bose-Newcomb:1974}, magnetic
resonance imaging \cite{Schultz-Seidel:2008,Qi-Wang:2008}, spectral
hypergraph theory \cite{Cooper-Dutle:2012,Hu-Qi:2012}, higher order Markov
chains \cite{Chang-Zhang:2013,Gleich-Lim-Yu:2015,Li-Ng:2014}, etc.

Unlike the matrix eigenvalue problem, computing eigenvalues of a general
higher-order tensor is NP-hard \cite{Hillar-Lim:2013}. Recently, Chen, Han
and Zhou \cite{Chen-Han-Zhou:2016} proposed a homotopy continuation method
for finding all eigenpairs of a general tensor. For the tensors of a certain
type, such as symmetric or nonnegative tensors, there are several algorithms
for computing one or some eigenpairs (including Z-eigenpair and H-eigenpair) .

For the computation of Z-eigenpairs, Kolda and Mayo \cite{Kolda-Mayo:2011}
proposed a shifted symmetric higher-order power method (SS-HOPM) for real symmetric tensors.
Gleich, Lim, and Yu \cite{Gleich-Lim-Yu:2015} proposed,  a always-stochastic Newton iteration for finding nonnegative Z-eigenpair of nonnegative tensors arising in a multilinear PageRank problem.

For the computation of  H-eigenpairs, Ng, Qi, and Zhou \cite{Ng-Qi-Zhou:2009} proposed a power-type method, NQZ algorithm, for the largest H-eigenvalue of weakly primitive nonnegative tensors. Some modeled versions of the power-type
method have been proposed in \cite%
{Liu-Zhou-Ibrahim:2010,Zhang-Qi-Xu:2012,Zhou-Qi-Wu:2013}. Recently, Liu, Guo
and Lin \cite{Liu-Guo-Lin:2016,Liu-Guo-Lin:2017} proposed a Newton-Noda
iteration (NNI) for finding the largest H-eigenvalue of  weakly
irreducible nonnegative tensors.

For a high-order nonnegative tensor $\mathcal{A}\in \mathbb{R}^{[m,n]}$,  $\mathcal{A}$ has nonnegative Z-eigenpairs  and H-eigenpairs,  but they are not unique (see \cite{Chang-Pearson-Zhang:2013,Chang-Qi-Zhang:2013,Chang-Pearson-Zhang:2008,Chang-Qi-Zhang:2013}). In many applications \cite{Chang-Zhang:2013,Gleich-Lim-Yu:2015,Cooper-Dutle:2012,Hu-Qi:2012,Li-Ng:2014}, computing the nonnegative Z-/H-eigenpairs is an important subject. Therefore, a central concern is how to avoid computing all the eigenvalues to find a few nonnegative Z-eigenpairs and H-eigenpairs. SS-HOPM \cite{Kolda-Mayo:2011,Gleich-Lim-Yu:2015} and NQZ \cite{Ng-Qi-Zhou:2009} can be used to compute a nonnegative Z-eigenpair and H-eigenpair, respectively, but the convergence may be quite slow. The always-stochastic Newton's method \cite{Gleich-Lim-Yu:2015} is a fast-converging algorithm when the starting point is sufficiently close to a solution. However, it's interestingly enough that the authors \cite{Gleich-Lim-Yu:2015} also provided a nonnegative tensor with a unique nonnegative Z-eigenpair, and the always-stochastic Newton's method fails to find it. Based on the reasons above mentioned, this motivates us to develop a continuation method to ensure the global convergence for nonnegative Z-eigenpairs.
The main contributions of this article are highlighted in the
following items.

{\leftmargini=7mm
\leftmarginii=6mm
\begin{itemize}
\item For \textit{nonnegative Z-eigenpairs}: we construct a linear homotopy $H_{\rm Z}(\mathbf{x},\lambda,t)=\mathbf{0}$, $t\in [0,1]$, where $H_{\rm Z}(\mathbf{x},\lambda,0)=\mathbf{0}$ has only one positive solution, $(\mathbf{x}_0,\lambda_0)$, and all real solutions of $H_{\rm Z}(\mathbf{x},\lambda,1)=\mathbf{0}$ are Z-eigenpairs of $\mathcal{A}$.
\begin{itemize}
\item[1.] We show that the solution curve of $H_{\rm Z}(\mathbf{x},\lambda,t)=
\mathbf{0}$ with initial $(\mathbf{x}_0,\lambda_0,0)$ is smooth.
Furthermore, we also show that the solution curve will reach a nonnegative
solution of $H_{\rm Z}(\mathbf{x},\lambda,1)=\mathbf{0}$ if all nonnegative
solutions of $H_{\rm Z}(\mathbf{x},\lambda,1)=\mathbf{0}$ are isolated (see
Theorems \ref{thm2.6} and \ref{thm2.8}). Hence, in this case, homotopy
continuation method is guaranteed to compute the nonnegative Z-eigenpair
of $\mathcal{A}$.

\item[2.] If $\mathcal{A}$ is irreducible and all nonnegative solutions of $
H_{\rm Z}(\mathbf{x},\lambda,1)=\mathbf{0}$ are isolated, then we show that the
number of positive Z-eigenpairs of $\mathcal{A}$, counting multiplicities,
is $2k+1$ for some integer $k\geqslant 0$ (see Corollary \ref{cor2.11}).

\item[3.] We propose a novel homotopy continuation method to compute an odd
number of positive Z-eigenpairs for an irreducible nonnegative tenor $
\mathcal{A}$ (see the flowchart in Figure \ref{fig1}).
\end{itemize}

\item For \textit{nonnegative H-eigenpairs}: we construct
a linear homotopy $H_{\rm H}(\mathbf{x},\lambda,t)=\mathbf{0}$, $t\in [0,1]$, where $H_{\rm H}(\mathbf{x},\lambda,0)=\mathbf{0}$ has only one positive solution, $(\mathbf{x}_0,\lambda_0)$, and all real solutions of $H_{\rm H}(\mathbf{x},\lambda,1)=\mathbf{0}$ are H-eigenpairs of $\mathcal{A}$. We show that
the solution curve of $H_{\rm H}(\mathbf{x},\lambda,t)=\mathbf{0}$ with initial $(
\mathbf{x}_0,\lambda_0,0)$ can be parameterized by $t\in [0,1)$. If the
nonnegative solutions of $H_{\rm H}(\mathbf{x},\lambda,1)=\mathbf{0}$ are
isolated, then the solution curve will reach a nonnegative solution of $H_{\rm H}(
\mathbf{x},\lambda,1)=\mathbf{0}$ (see Theorem \ref{thm2.3}), and hence,
homotopy continuation method is guaranteed to compute the nonnegative H-eigenpair of $\mathcal{A}$. Note that if $\mathcal{A}$ is weakly
irreducible, then $H_{\rm H}(\mathbf{x},\lambda,1)=\mathbf{0}$ has only one
positive isolated solution (see Theorem \ref{thm2.4}).
\end{itemize}}


This paper is organized as follows. The notations and preliminary results
are in Section 2. In Section 3, we develop homotopy continuation methods to
compute the nonnegative Z-eigenpairs and H-eigenpair of a nonnegative
tensor $\mathcal{A}$ and show that the continuation methods are guaranteed
to compute the nonnegative eigenpairs. In Section 4, we propose a novel
homotopy continuation method to compute an odd number of positive Z-eigenpairs for an irreducible nonnegative tenor. Some numerical results are
presented in Section 5. Conclusion of this paper is given in Section 6.
%

\section{Preliminaries}

Let $\mathbb{F}=\mathbb{C}$ or $\mathbb{R}$ be the complex field or the real
field. An $m$th-order rank-1 tensor $\mathcal{A}=[\mathcal{A}_{i_{1},i_{2},\cdots ,i_{m}}]\in \mathbb{F}^{n_{1}\times n_{2}\times \cdots
\times n_{m}}$ is defined as the outer product of $m$ nonzero vectors $
\mathbf{u}_{k}\in \mathbb{F}^{n_{k}}$ for $k=1,\cdots ,m$, denoted by $
\mathbf{u}_{1}\circ \mathbf{u}_{2}\circ \ldots \circ \mathbf{u}_{m}$. That
is,
\begin{equation*}
\mathcal{A}_{i_{1},i_{2},\cdots ,i_{m}}=u_{1,i_{1}}u_{2,i_{2}}\cdots
u_{m,i_{m}},
\end{equation*}
where $u_{k,i_{k}}$ is the $i_{k}$-th component of vector $\mathbf{u}_{k}$.
The \textit{$k$-mode product} of a tensor $\mathcal{A}\in \mathbb{F}^{n_{1}\times n_{2}\times \cdots \times n_{m}}$ with a vector $\mathbf{x}=(x_1,\cdots, x_{n_k})^{\top}\in \mathbb{F}^{n_k}$ is denoted by $\mathcal{A}\times_{k}\mathbf{x}$ and is $(m-1)$th-order tensor with size $n_{1}\times
\cdots\times n_{k-1}\times n_{k+1}\times \cdots \times n_{m}$. Elementwise,
we have
\begin{align*}
(\mathcal{A}\times_{k}\mathbf{x})_{i_1,\ldots,i_{k-1},i_{k+1},\ldots,i_m}=
\sum_{i_k=1}^{n_k}\mathcal{A}_{i_{1},i_{2},\cdots ,i_{m}}x_{i_k}.
\end{align*}
For a tensor $\mathcal{A}\in \mathbb{F}^{[m,n]}$ and a vector $\mathbf{x}
=(x_1,\cdots, x_{n})^{\top}\in \mathbb{F}^{n}$, we denote $\mathcal{A}\mathbf{x}^{m-1}=\mathcal{A}\times_{2}\mathbf{x}\times_{3}\cdots\times_{m}\mathbf{x}$ and $\mathbf{x}^{[\ell]}=(x_1^{\ell},\cdots, x_{n}^{\ell})$,
where $\ell$ is a positive real number.

Let $\mathbb{R}^{[m,n]}_{\geqslant 0}$ ($\mathbb{R}^{[m,n]}_{> 0}$
) denote the set of all real nonnegative (positive) $m$th-order $n$-dimensional tensors. We use calligraphic letters to denote tensors, capital
letters to denote matrices and lowercase (bold) letters to denote scalars
(vectors). For a tensor $\mathcal{A}$, $\mathcal{A}\geqslant 0$ ($\mathcal{A}> 0$) denotes a nonnegative (positive) tensor with nonnegative (positive)
entries. A real square nonsingular $M$-matrix $B$ can be written as $sI-A$
with $A\geq 0$ if $s>\rho (A)$, and a singular $M$-matrix if $s=\rho (A)$,
where $\rho (\cdot )$ is the spectral radius.
We use the $2$-norm for vectors and matrices, and all vectors are $n$%
-vectors and all matrices are $n\times n$, unless specified otherwise.

\subsection{Tensor eigenvalues and eigenvectors}
The following definition of Z-eigenvalues and H-eigenvalues was introduced by Qi
in \cite{Qi:2005,Qi:2007}.
\begin{Definition}
\label{def2.2} Suppose that $\mathcal{A}$ is an $m$th-order $n$-dimensional tensor.
\begin{itemize}
\item[(i)] $\lambda \in \mathbb{R}$ is called a Z-eigenvalue of $\mathcal{A}$ with the corresponding Z-eigenvector $\mathbf{x} \in \mathbb{R}^n\backslash\{ \mathbf{0}\}$ (or $(\lambda, \mathbf{x})$ is a Z-eigenpair)
if $(\lambda, \mathbf{x})$ satisfies
\begin{align}  \label{Zeig}
\mathcal{A}\mathbf{x}^{m-1}=\lambda \mathbf{x}, \text{ with } \|\mathbf{x}\|=1.
\end{align}

\item[(ii)] $\lambda \in \mathbb{R}$ is called a H-eigenvalue of $\mathcal{A}$ with the corresponding H-eigenvector $\mathbf{x} \in \mathbb{R}^n\backslash\{ \mathbf{0}\}$ (or $(\lambda, \mathbf{x})$ is a H-eigenpair)
if $(\lambda, \mathbf{x})$ satisfies
\begin{align}  \label{Heig}
\mathcal{A}\mathbf{x}^{m-1}=\lambda \mathbf{x}^{[m-1]}.
\end{align}
\end{itemize}
\end{Definition}

In \cite{Chen-Han-Zhou:2016,Ni-Qi-Wang-Wang:2007}, the authors proved that if the tensor $\mathcal{A}$ is generic, then the
number of isolated solutions of \eqref{Zeig} and of \eqref{Heig} with $\|\mathbf{x}\|=1$ are exactly $\frac{(m-1)^{n}-1}{m-2}$ and $n(m-1)^{n-1}$,
respectively.

The Perron-Frobenius theorems for (weakly) irreducible nonnegative tensor
have been widely investigated. The definition of (weakly) irreducible
tensor was introduced in \cite{HHQ}.

\begin{Definition}
\label{def2.4} Suppose that $\mathcal{A}$ is an $m$th-order $n$-dimensional tensor.
\begin{itemize}
\item[(i)] $\mathcal{A}$ is called reducible if there exists a nonempty proper subset $S\subset \{1,2,\cdots,n\}$ such that
\begin{align*}
\mathcal{A}_{i_1,i_2,\cdots,i_m}=0,\ \forall i_1\in S,\ \forall i_2,\ldots, i_m \notin S.
\end{align*}
If $\mathcal{A}$ is not reducible, then  $\mathcal{A}$ is called irreducible.

\item[(ii)] $\mathcal{A}$ is called weakly irreducible if for every nonempty
subset $S\subset \{1,2,\cdots,n\}$ there exist $i_1\in S$ and $i_2,\cdots, i_m$
with at least one $i_q \notin S$, $q=2,\cdots,m$ such that $\mathcal{A}_{i_1,i_2,\cdots,i_m}\neq 0$.
\end{itemize}
\end{Definition}

Note that when $m=2$, the definitions of an irreducible tensor and a weakly irreducible tensor are the same as the definition of an irreducible matrix. From the definitions, it is easily seen that if $\mathcal{A}$ is irreducible then $\mathcal{A}$ is weakly irreducible.

The existence of nonnegative Z-eigenpair (see \cite{Chang-Pearson-Zhang:2013,Chang-Qi-Zhang:2013}) or H-eigenpair (see \cite{Chang-Pearson-Zhang:2008,Chang-Qi-Zhang:2013}) of a nonnegative tensor $\mathcal{A}$ have been investigated. They satisfy the following properties.

{\leftmargini=7mm
\leftmarginii=9mm
\begin{itemize}
\item \textbf{Z-eigenpair}: Let $\mathcal{A}\in \mathbb{R}_{\geqslant 0}^{[n,m]}$
and $\mathcal{Z}(\mathcal{A})$ be the set of all Z-eigenvalues of $\mathcal{A}$. Then $\mathcal{A}$ has Z-eigenpair $(\lambda_0,\mathbf{x}_0)\in \mathbb{R}_{\geqslant 0}\times \mathbb{R}_{\geqslant 0}^{n}$, i.e., $\mathcal{Z}(\mathcal{A})\neq \emptyset$. In fact, the set $\mathcal{Z}(\mathcal{A})$ is not necessarily a finite set in general (see Example 3.6 in
\cite{Chang-Pearson-Zhang:2013}). The set $\mathcal{Z}(\mathcal{A})$ and the
Z-eigenpair $(\lambda_0,\mathbf{x}_0)$ of $\mathcal{A}$ satisfy the
following statements:

\begin{itemize}
\item[1.] The set $\mathcal{Z}(\mathcal{A})$ is bounded. It follows from
Proposition 3.3 of \cite{Chang-Pearson-Zhang:2013} that
\begin{align}  \label{equ_bound}
\varrho(\mathcal{A})\equiv \mathrm{sup}\{|\lambda|\ |\ \lambda\in\mathcal{Z}(\mathcal{A})\}\leqslant \max_{1\leqslant i\leqslant n}\sqrt{n}\sum_{i_2,\cdots,i_m=1}^{n}\mathcal{A}_{i,i_2,\cdots,i_m}.
\end{align}

\item[2.] If $\mathcal{A}$ is irreducible, then $\lambda_0> 0$ and $\mathbf{x}_0> 0$.

\item[3.] If $\mathcal{A}$ is weakly symmetric,\footnote{A tensor $\mathcal{A}\in \mathbb{R}^{[m,n]}$ is called weakly symmetric if
it satisfies $\frac{d}{d\mathbf{x}}\mathcal{A}\mathbf{x}^{m}=m\mathcal{A}\mathbf{x}^{m-1}$.} then the cardinality of $\mathcal{Z}(\mathcal{A})$ is
finite,
\begin{align*}
\varrho(\mathcal{A})\in \mathcal{Z}(\mathcal{A}) \text{ and }\varrho(\mathcal{A})=\max_{\mathbf{x}\in \mathbb{R}^{n}, \ \|\mathbf{x}\|=1}\mathcal{A}\mathbf{x}^m.
\end{align*}
\end{itemize}
\item \textbf{H-eigenpair}: If $\mathcal{A}\in \mathbb{R}_{\geqslant 0}^{[n,m]}$
then $\mathcal{A}$ has H-eigenpairs $(\lambda_0,\mathbf{x}_0)\in \mathbb{R}_{\geqslant 0}\times \mathbb{R}_{\geqslant 0}^{n}$. Suppose that $\mathcal{A}
$ is weakly irreducible then the eigenpair $(\lambda_0,\mathbf{x}_0)$
satisfies the following statements:

\begin{itemize}
\item[1.] $\lambda_0>0$ and $\mathbf{x}_0>0$.

\item[2.] If $\lambda$ is an eigenvalue with nonnegative eigenvector, then $\lambda=\lambda_0$. Moreover, the nonnegative eigenvector is unique up to a
multiplicative constant.

\item[3.] If $\lambda$ is an eigenvalue then $|\lambda|\leqslant \lambda_0.$
\end{itemize}
\end{itemize}}


The following lemma is straightforward.

\begin{Lemma}
\label{lem2.1} Let $\mathcal{A}\in \mathbb{R}^{[m,n]}_{>0}$. Then $\mathcal{A}$ has no Z-eigenpair (or H-eigenpair) on $\partial(\mathbb{R}^{n+1}_{\geqslant 0})$, where $\partial(\mathbb{R}^{n+1}_{\geqslant 0})$ is
the boundary of $\mathbb{R}^{n+1}_{\geqslant 0}$.
\end{Lemma}
\begin{proof}
Assume that $(\lambda,\mathbf{x})\in \partial(\mathbb{R}^{n+1}_{\geqslant 0})$ is a Z-eigenpair (or H-eigenpair) of $\mathcal{A}\in \mathbb{R}^{[m,n]}_{>0}$.
Suppose that $\mathbf{x}\in \partial(\mathbb{R}^{n}_{\geqslant 0})$, then there exists $i\in \{1,2,\ldots,n\}$ such that $x_{i}=0$, where $x_{i}$ is $i$th component of $\mathbf{x}$. Then the $i$th component of the vector $\mathcal{A}\mathbf{x}^{m-1}$ is zero, i.e., $\sum_{i_2,\cdots,i_m=1}^{n}\mathcal{A}_{i,i_2,\cdots,i_m}x_{i_2}\cdots x_{i_m}=0$. This is a contradiction because $\mathcal{A}_{i,i_2,\cdots,i_m}>0$ and  $\mathbf{x}\in \mathbb{R}_{\geqslant 0}$ is a nonzero vector. Hence, $\mathbf{x}>0$. Since $\mathcal{A}>0$ and $\mathbf{x}>0$, we have $\lambda>0$. Hence, $(\lambda,\mathbf{x})\notin   \partial(\mathbb{R}^{n+1}_{\geqslant 0})$.
\end{proof}

Next, we describe all Z-eigenpairs and H-eigenpairs of a rank-1
nonnegative tensor.

\begin{Lemma}
\label{lem2.2} Let $\mathcal{A}_0=\mathbf{x}_1\circ\cdots \circ\mathbf{x}_m\in \mathbb{R}^{[m,n]}_{\geqslant 0}$, where $\mathbf{x}_1,\ldots, \mathbf{x}_m\in \mathbb{R}_{\geqslant 0}^n$
are nonzero vectors. Then
\begin{itemize}
\item[(i)] Z-eigenpairs: Let $\mathbf{x}_0=\frac{\mathbf{x}_1}{\|\mathbf{x}_1\|}$ and $\lambda_0=\|\mathbf{x}_1\|\prod_{k=2}^m\left(\mathbf{x}_k^{\top}\mathbf{x}_0\right)$. Then $\left(\lambda_0,\mathbf{x}_0\right)$, $\left((-1)^m\lambda_0,-\mathbf{x}_0\right)$ and $(0,\mathbf{w})$ with $\mathbf{w}\in \bigcup_{k=2}^{m}\mathrm{span}\{\mathbf{x}_k\}^{\bot}$, $\|\mathbf{w}\|=1$ are Z-eigenpairs of $\mathcal{A}_0$. In addition, if $\mathbf{x}_1$ is a positive vector, then the eigenvalue $\lambda_0>0$ with
positive eigenvector $\mathbf{x}_0$. Furthermore, if $\mathbf{x}_1,\ldots,
\mathbf{x}_m\in \mathbb{R}_{> 0}^n$, then $\mathbf{x}_0\in \mathbb{R}_{>0}^n$ is
the unique nonnegative eigenvector of $\mathcal{A}_0$.

\item[(ii)] H-eigenpairs: Let $\lambda_0=\prod_{k=2}^m\left(\mathbf{x}_k^{\top}\mathbf{x}_1^{[1/(m-1)]}\right)$ and $\mathbf{x}_0=\frac{\mathbf{x}_1^{[1/(m-1)]}}{\|\mathbf{x}_1^{[1/(m-1)]}\|}$. Then $\left(\lambda_0,\mathbf{x}_0\right)$ and $(0,\mathbf{w})$ with $\mathbf{w}\in
\bigcup_{k=2}^{m}\mathrm{span}\{\mathbf{x}_k\}^{\bot}$, $\mathbf{w}\neq
\mathbf{0}$ are H-eigenpairs of $\mathcal{A}_0$. In addition, if $\mathbf{x}_1$ is a positive vector, then the eigenvalue $\lambda_0>0$ with positive
eigenvector $\mathbf{x}_0$. Furthermore, if $\mathbf{x}_1,\ldots, \mathbf{x}_m\in \mathbb{R}_{> 0}^n$, then $c\mathbf{x}_0\in \mathbb{R}^{n}_{> 0}$ with
$c>0$ is the unique nonnegative eigenvector of $\mathcal{A}_0$.
\end{itemize}
\end{Lemma}

\begin{proof}
$(i)$ Suppose that $(0,\mathbf{w})$  is a Z-eigenpair of tensor $\mathcal{A}_0=\mathbf{x}_1\circ\cdots \circ\mathbf{x}_m$. Then
\begin{align*}
\mathcal{A}_0\mathbf{w}^{m-1}=(\mathbf{x}_2^{\top}\mathbf{w})\cdots(\mathbf{x}_m^{\top}\mathbf{w})\mathbf{x}_1=\mathbf{0}.
\end{align*}
Since $\mathbf{x}_1\neq \mathbf{0}$, we obtain that $\prod_{k=2}^{m}(\mathbf{x}_k^{\top}\mathbf{w})=0$ and hence, there exists a $k\in \{2,\ldots,m\}$ such that  $\mathbf{x}_k^{\top}\mathbf{w}=0$. So, the Z-eigenvector $\mathbf{w}\in \bigcup_{k=2}^{m}{\rm span}\{\mathbf{x}_k\}^{\bot}$ and $\|\mathbf{w}\|=1$. Suppose that $(\lambda,\mathbf{w})$ with $\lambda\neq 0$ is a Z-eigenpair of tensor $\mathcal{A}_0$. Then
\begin{align*}
\lambda\mathbf{w}=\mathcal{A}_0\mathbf{w}^{m-1}=\left(\prod_{k=2}^m\left(\mathbf{x}_k^{\top}\mathbf{w}\right)\right)\mathbf{x}_1.
\end{align*}
Since the Z-eigenvector is a unit vector, we obtain that  $\mathbf{w}=\mathbf{x}_0$ or $\mathbf{w}=-\mathbf{x}_0$ is a Z-eigenvector corresponding to Z-eigenvalue $\lambda=\lambda_0$ or $\lambda=(-1)^m\lambda_0$, respectively, where   $\mathbf{x}_0=\frac{\mathbf{x}_1}{\|\mathbf{x}_1\|}$ and $\lambda_0=\|\mathbf{x}_1\|\prod_{k=2}^m\left(\mathbf{x}_k^{\top}\mathbf{x}_0\right)$. If $\mathbf{x}_1>0$, then it is easily seen that $\lambda_0>0$.  Furthermore, if $\mathbf{x}_1,\cdots,\mathbf{x}_m\in \mathbb{R}^n_{>0}$, then $ \left(\bigcup_{k=2}^{m}{\rm span}\{\mathbf{x}_k\}^{\bot}\right)\bigcap \mathbb{R}_{\geqslant 0}^n=\{\mathbf{0}\}$ and hence $\mathbf{x}_0$ is the unique nonnegative eigenvector of $\mathcal{A}_0$.

$(ii)$ Similarly, suppose that  $(0,\mathbf{w})$  is a H-eigenpair of tensor $\mathcal{A}_0=\mathbf{x}_1\circ\cdots \circ\mathbf{x}_m$. Then
we obtain the H-eigenvector $\mathbf{w}\in \bigcup_{k=2}^{m}{\rm span}\{\mathbf{x}_k\}^{\bot}$  and $\mathbf{w}\neq \mathbf{0}$.
Suppose that $(\lambda,\mathbf{w})$ with $\lambda\neq 0$ is a H-eigenpair of tensor $\mathcal{A}_0$. Then
\begin{align*}
\lambda \mathbf{w}^{[m-1]}=\mathcal{A}_0\mathbf{w}^{m-1}=\left(\prod_{k=2}^m(\mathbf{x}_k^{\top}\mathbf{w})\right)\mathbf{x}_1.
\end{align*}
Since $\lambda\neq 0$, we have  $\mathbf{w}=c\mathbf{x}_0=c\frac{\mathbf{x}_1^{[1/(m-1)]}}{\|\mathbf{x}_1^{[1/(m-1)]}\|}$ and  $\lambda=\lambda_0\equiv \prod_{k=2}^m\left(\mathbf{x}_k^{\top}\mathbf{x}_1^{[1/(m-1)]}\right)$, where $c\neq 0$.  The eigenvalue $\lambda\geqslant 0$, because $\mathbf{x}_1,\ldots, \mathbf{x}_m\in \mathbb{R}_{\geqslant 0}^n$. Since $\lambda\neq 0$,  we obtain that  $\mathbf{x}_k^{\top}\mathbf{x}_1^{[1/(m-1)]}\neq 0$ for each $k\in \{2,\cdots,m\}$.
If $\mathbf{x}_1>0$, then $\lambda_0=\prod_{k=2}^m\left(\mathbf{x}_k^{\top}\mathbf{x}_1^{[1/(m-1)]}\right)>0$  because $\mathbf{x}_2,\ldots, \mathbf{x}_m\in \mathbb{R}_{\geqslant 0}^n$ are nonzero. Furthermore, if $\mathbf{x}_1,\cdots,\mathbf{x}_m\in \mathbb{R}^n_{>0}$, then $ \left(\bigcup_{k=2}^{m}{\rm span}\{\mathbf{x}_k\}^{\bot}\right)\bigcap \mathbb{R}_{\geqslant 0}^n=\{\mathbf{0}\}$. So, $c\mathbf{x}_0$ with $c>0$ is the unique nonnegative eigenvector of  $\mathcal{A}_0$. This completes the proof.
\end{proof}

\subsection{The basic theorems of continuation methods}

In the following, we will introduce some preliminary theorems which are
useful in study of continuation methods.

\begin{Definition}
\label{def2.5} Let $H:\mathbb{R}^{n}\rightarrow \mathbb{R}^{k}$ be a
continuously differentiable function (denoted by $H\in C^1(\mathbb{R}^n)$).
A point $\mathbf{p}\in \mathbb{R}^{k}$ is called regular value if $\mathrm{rank}(\mathscr{D}_{\mathbf{x}}H(\mathbf{x}_*))=\min\{n,k\}$ for all $\mathbf{x}_*\in H^{-1}(\mathbf{p})\subseteq \mathbb{R}^{n}$, where $\mathscr{D}_{\mathbf{x}}H(\mathbf{x})$ denotes the partial derivatives of $H(\mathbf{x}).$
\end{Definition}

Now, we state the Parameterized Sard's Theorem. The proof can be
found in \cite{Chow-Mallet-Paret-Yorke :1978}.

\begin{Theorem}[Parameterized Sard's Theorem \protect\cite{Chow-Mallet-Paret-Yorke :1978}]
\label{thm2.6.5} Let $U\subseteq \mathbb{R}^{n}$ and $V\subseteq \mathbb{R}^{q}$ be open sets, and $P:U\times V\rightarrow \mathbb{R}^{k}$ be a smooth
map. If $\mathbf{0}\in \mathbb{R}^{k}$ is a regular value of $P$, then for
almost all $\mathbf{c}\in V$, $\mathbf{0}$ is a regular value of $H(\cdot)\equiv P(\cdot,\mathbf{c})$.
\end{Theorem}

Suppose that $\mathbf{p}\in \mathbb{R}^n$ is a regular value of a
continuously differentiable function $F:\mathbb{R}^n\rightarrow \mathbb{R}^n$, $\Omega\subseteq \mathbb{R}^n$ is open bounded and $\mathbf{p}\notin
F(\partial\Omega)$. Then the set $F^{-1}(\mathbf{p})\cap\overline{\Omega}$
is a finite set (see Lemma 3.5 in \cite{ Keller:1987}). The ``degree" of $F:\overline{\Omega}\rightarrow \mathbb{R}^n$ at a point $\mathbf{p}\in \mathbb{R}^n$ plays an important role in investigating the solution of $F(\mathbf{x})=\mathbf{p}$. The definition of the degree for $F$ is as follows.

\begin{Definition}[See \protect\cite{ Keller:1987}]
\label{def2.6} Let $F:\overline{\Omega}\rightarrow \mathbb{R}^{n}$ be a
continuously differentiable function, where $\Omega\subseteq \mathbb{R}^n$
is an open bounded subset. Let $\mathbf{p}\in \mathbb{R}^n$ be a regular
value of $F$ and $\mathbf{p}\notin F(\partial\Omega)$. The the degree of $F$
on $\Omega$ for $\mathbf{p}$ is defined as:
\begin{align*}
\mathrm{deg }(F,\Omega,\mathbf{p})=\sum_{\mathbf{x}\in F^{-1}(\mathbf{p})\cap\overline{\Omega}}\mathrm{Sgn}[\mathrm{det} \mathscr{D}_{\mathbf{x}}F(\mathbf{x})].
\end{align*}
\end{Definition}

\begin{Remark}
\label{rem2.1} The conditions of the definition for the degree of $F$ on $\Omega$
for $\mathbf{p}$ can be relaxed. It only requires the function $F$ satisfies
$(i)$ $F\in C^1(\overline{\Omega})$ and $(ii)$ $\mathbf{p}\notin
F(\partial\Omega)$ (see Definition 3.19 in \cite{Keller:1987}). That is, the
condition, $\mathbf{p}$ is a regular value of $F$, can be omitted in the
definition. The main idea is that if $\mathbf{p}$ is not a regular value
then the degree can be defined as $\mathrm{deg}(F,\Omega,\mathbf{p})\equiv
\mathrm{deg}(F,\Omega,\mathbf{q})$, where $\mathbf{q}$ is a regular value of
$F$ and $\|\mathbf{q}-\mathbf{p}\|<\inf_{\mathbf{x}\in \partial\Omega}\|F(\mathbf{x})-\mathbf{p}\|$.
\end{Remark}

\begin{Theorem}[Homotopy Invariance of Degree, see \protect\cite{Keller:1987}]
\label{thm2.9} Let $H:\mathbb{R}^{n}\times [0,1]\rightarrow \mathbb{R}^{n}$,
$\Omega\subset \mathbb{R}^n$ bounded open set and $\mathbf{p}\in \mathbb{R}^n
$ satisfy:

\begin{itemize}
\item[(i)] $H\in C^2(\overline{\Omega}\times [0,1])$ and $H(\mathbf{u}
,t)\neq \mathbf{p}$ on $\partial\Omega\times[0,1]$,

\item[(ii)] $F^0(\mathbf{u})\equiv H(\mathbf{u},0)$, $F(\mathbf{u})\equiv H(
\mathbf{u},1)$,

\item[(iii)] $\mathbf{p}$ is a regular value for $H$ on $\overline{\Omega}
\times [0,1]$ and for $F^0$ and $F$ on $\overline{\Omega}$.
\end{itemize}
Then $\mathrm{deg}(H(\cdot,t),\Omega,\mathbf{p})$ is independent of $t\in
[0,1]$. In particular, $\mathrm{deg}(F^0,\Omega,\mathbf{p})=\mathrm{deg}(F,\Omega,\mathbf{p})$.
\end{Theorem}

\begin{Remark}
\label{rem2.2} Without the restriction (i) in Theorem \ref{thm2.9}, the
theorem is true for the weakly definition of degree (see Remark \ref{rem2.1}).
\end{Remark}

\subsection{Linear homotopies}

Given a tensor $\mathcal{A}\in \mathbb{R}^{[m,n]}_{\geqslant 0}$, let
\begin{align}  \label{eq2.1}
\mathcal{A}_0=\mathbf{x}_1\circ\cdots \circ\mathbf{x}_m\in \mathbb{R}^{[m,n]}_{> 0},
\end{align}
be a rank-1 tensor, where $\mathbf{x}_1, \ldots, \mathbf{x}_m\in \mathbb{R}^n_{>0}$ are generic. We define some systems of polynomial equations:
\begin{align}  \label{eqZ}
F^0_{\rm Z}(\mathbf{x},\lambda)=\left(
\begin{array}{c}
\mathcal{A}_0\mathbf{x}^{m-1}-\lambda \mathbf{x} \\
\mathbf{x}^{\top}\mathbf{x}-1
\end{array}
\right)=\mathbf{0},\ \ \ F_{\rm Z}(\mathbf{x},\lambda)=\left(
\begin{array}{c}
\mathcal{A}\mathbf{x}^{m-1}-\lambda \mathbf{x} \\
\mathbf{x}^{\top}\mathbf{x}-1
\end{array}
\right)=\mathbf{0},
\end{align}
and
\begin{align}  \label{eqH}
F^0_{\rm H}(\mathbf{x},\lambda)=\left(
\begin{array}{c}
\mathcal{A}_0\mathbf{x}^{m-1}-\lambda \mathbf{x}^{[m-1]} \\
\mathbf{x}^{\top}\mathbf{x}-1
\end{array}
\right)=\mathbf{0},\ F_{\rm H}(\mathbf{x},\lambda)=\left(
\begin{array}{c}
\mathcal{A}\mathbf{x}^{m-1}-\lambda \mathbf{x}^{[m-1]} \\
\mathbf{x}^{\top}\mathbf{x}-1
\end{array}
\right)=\mathbf{0}.
\end{align}
From Definition \ref{def2.2}, if the H-eigenvector of $\mathcal{A}$ is a unit vector, then the Z-eigenpair and H-eigenpair
should satisfy $F_{\rm Z}(\mathbf{x},\lambda)=\mathbf{0}$ and $F_{\rm H}(\mathbf{x}
,\lambda)=\mathbf{0}$, respectively. Let
\begin{align}  \label{eqAt}
\mathcal{A}(t)=(1-t)\mathcal{A}_0+t\mathcal{A}.
\end{align}
Now, we consider two linear homotopies
\begin{align}  \label{eq2.7}
H_{\rm Z}(\mathbf{x},\lambda, t)=\left(
\begin{array}{c}
\mathcal{A}(t)\mathbf{x}^{m-1}-\lambda \mathbf{x} \\
\mathbf{x}^{\top}\mathbf{x}-1
\end{array}
\right)=\mathbf{0},\ \ \ \text{ for }t\in[0,1]
\end{align}
and
\begin{align}  \label{eq2.2}
H_{\rm H}(\mathbf{x},\lambda, t)=\left(
\begin{array}{c}
\mathcal{A}(t)\mathbf{x}^{m-1}-\lambda \mathbf{x}^{[m-1]} \\
\mathbf{x}^{\top}\mathbf{x}-1
\end{array}
\right)=\mathbf{0},\ \ \ \text{ for }t\in[0,1].
\end{align}
It is easily seen that $H_{\rm Z}(\mathbf{x},\lambda, 0)=F^0_{\rm Z}(\mathbf{x},\lambda)$,
$H_{\rm Z}(\mathbf{x},\lambda, 1)=F_{\rm Z}(\mathbf{x},\lambda)$, $H_{\rm H}(\mathbf{x},\lambda, 0)=F^0_{\rm H}(\mathbf{x},\lambda)$ and $H_{\rm H}(\mathbf{x},\lambda, 1)=F_{\rm H}(
\mathbf{x},\lambda)$.

We then have the following results.
\begin{Theorem}
\label{thm2.1} For any $t\in [0,1)$, $H_{\rm Z}(\mathbf{x},\lambda, t)\neq \mathbf{0}$ and $H_{\rm H}(\mathbf{x
},\lambda, t)\neq \mathbf{0}$ on the boundary of $\mathbb{R}
^{n+1}_{\geqslant 0}$.
\end{Theorem}

\begin{proof}
For any $t\in [0,1)$, $\mathcal{A}(t)=(1-t)\mathcal{A}_0+t\mathcal{A}>0$ because $\mathcal{A}_0>0$ and  $\mathcal{A}\geqslant 0$. It follows from Lemma \ref{lem2.1} that $H_{\rm Z}(\mathbf{x},\lambda, t)\neq \mathbf{0}$ and $H_{\rm H}(\mathbf{x},\lambda, t)\neq \mathbf{0}$  on $\partial(\mathbb{R}^{n+1}_{\geqslant 0})$.
\end{proof}

\section{Homotopy Continuation Methods}

In this section, we propose homotopy continuation methods for computing
nonnegative  Z-/H-eigenpairs of a real nonnegative tensor $%
\mathcal{A}\in \mathbb{R}^{[m,n]}_{\geqslant 0}$. The following lemma is
useful in our later analysis.

\begin{Lemma}
\label{lem3.3} Suppose that $A\in \mathbb{R}^{n\times n}$ is an irreducible
singular $M$-matrix and $\mathbf{x},\mathbf{y}\in \mathbb{R}_{>0}^n$. Then
the matrix $\left[
\begin{array}{c|c}
A & \mathbf{x} \\ \hline
\mathbf{y}^{\top} & 0
\end{array}
\right]\in \mathbb{R}^{(n+1)\times (n+1)}$ is invertible.
\end{Lemma}

\begin{proof}
Suppose that there exists a nonzero vector $\mathbf{z}=(\mathbf{z}_1^{\top},z_2)^{\top}\in \mathbb{R}^{n+1}$ such that
\begin{align}\label{eq2.3}
\left[\begin{array}{c|c}A&\mathbf{x}\\\hline
\mathbf{y}^{\top}&0\end{array}\right]\left[\begin{array}{c}\mathbf{z}_1\\z_2
\end{array}\right]=\mathbf{0}.
\end{align}
Then $A\mathbf{z}_1=-z_2\mathbf{x}$. Since $A$ is an  irreducible singular $M$-matrix, there is a vector $\mathbf{w}\in \mathbb{R}_{>0}^n$ such that $\mathbf{w}^{\top}A=\mathbf{0}^{\top}$. Then $z_2(\mathbf{w}^{\top}\mathbf{x})=-\mathbf{w}^{\top}A\mathbf{z}_1=0$. Since $\mathbf{w},\mathbf{x}\in \mathbb{R}_{>0}^n$, $\mathbf{w}^{\top}\mathbf{x}>0$ and hence $z_2=0$. From \eqref{eq2.3}, we have $A\mathbf{z}_1=\mathbf{0}$ and $\mathbf{y}^{\top}\mathbf{z}_1=0$. Since $A$ is an  irreducible singular $M$-matrix, $\mathbf{z}_1>0$, and hence $\mathbf{y}^{\top}\mathbf{z}_1>0$, a contradiction. This completes the proof.
\end{proof}

\subsection{Computing the Z-eigenpair of nonnegative tensors}

Given a tensor $\mathcal{A}\in \mathbb{R}_{\geqslant 0}^{[m,n]}$, let $
\mathcal{A}_0=\mathbf{x}_1\circ\cdots\circ\mathbf{x}_1\in \mathbb{R}_{>
0}^{[m,n]}$ be a symmetric rank-1 tensor, where $\mathbf{x}_1\in \mathbb{R}^{n}_{>0}$ is generic. Let
\begin{align}  \label{eqlamxz}
\mathbf{x}_0=\frac{\mathbf{x}_1}{\|\mathbf{x}_1\|}>0,\ \ \ \ \ \lambda_0=\|
\mathbf{x}_1\|\left(\mathbf{x}_1^{\top}\mathbf{x}_0\right)^{m-1}=\|\mathbf{x}_1\|^m>0.
\end{align}
It follows from Lemma \ref{lem2.2} $(i)$ that $(\lambda_0,\mathbf{x}_0)\in \mathbb{
R}^{n+1}_{>0}$ is a Z-eigenpair of $\mathcal{A}_0$ and $\mathbf{x}_0$ is
the unique Z-eigenvector of $\mathcal{A}_0$ in $\mathbb{R}^{n}_{\geqslant
0}$. Then $F^0_{\rm Z}(\mathbf{x}_0,\lambda_0)=H_{\rm Z}( \mathbf{x}_0,\lambda_0, 0)=\mathbf{0}$, where $F^0_{\rm Z}$ is defined in \eqref{eqZ} and $H_{\rm Z}$ is defined in
\eqref{eq2.7} with the symmetric rank-1 tensor $\mathcal{A}_0$.

Suppose that $(\mathbf{x}_*,\lambda_*,t_*)\in \mathbb{R}^{n}_{> 0}\times
\mathbb{R}_{> 0}\times [0,1)$ is a solution of $H_{\rm Z}( \mathbf{x},\lambda, t)=
\mathbf{0}$. The Jacobian matrix of $H_{\rm Z}$ at $(\mathbf{x}_*,\lambda_*,t_*)$
has the form
\begin{subequations}
\label{eq2.8}
\begin{align}  \label{eq2.8a}
\mathscr{D}_{\mathbf{x},\lambda,t}H_{\rm Z}(\mathbf{x}_*,\lambda_*,t_*)=[
\mathscr{D}_{\mathbf{x},\lambda}H_{\rm Z}(\mathbf{x}_*,\lambda_*,t_*)\ |\
\mathscr{D}_{t}H_{\rm Z}(\mathbf{x}_*,\lambda_*,t_*)],
\end{align}
where
\begin{align}
&\mathscr{D}_{\mathbf{x},\lambda}H_{\rm Z}(\mathbf{x}_*,\lambda_*,t_*)=\left[
\begin{array}{c|c}
A_{t_*}-\lambda_* I_n & -\mathbf{x}_* \\ \hline
2\mathbf{x}_*^{\top} & 0
\end{array}%
\right]\in \mathbb{R}^{(n+1)\times (n+1)},  \label{eq2.8b} \\
&\mathscr{D}_{t}H_{\rm Z}(\mathbf{x}_*,\lambda_*,t_*)=\left[%
\begin{array}{c}
(\mathcal{A}-\mathcal{A}_0)\mathbf{x}_*^{m-1} \\ \hline
0%
\end{array}%
\right]\in \mathbb{R}^{n+1} \label{eq2.8c}
\end{align}
with
\end{subequations}
\begin{align}
A_{t_*}&\equiv \mathscr{D}_{\mathbf{x}}(\mathcal{A}(t_*)\mathbf{x}^{m-1})|_{%
\mathbf{x}=\mathbf{x}_*}\notag \\
&=\sum_{k=2}^m\mathcal{A}(t_*)\times_2\mathbf{x}%
_*\cdots \times_{k-1}\mathbf{x}_*\times_{k+1}\mathbf{x}_*\cdots \times_m%
\mathbf{x}_*\in \mathbb{R}^{n\times n}. \label{derAt}
\end{align}
Since $\mathbf{x}_*> 0$, it follows from \eqref{derAt} that $A_{t_*}>0$ and $%
A_{t_*}\mathbf{x}_*=(m-1)\mathcal{A}(t_*)\mathbf{x}_*^{m-1}=(m-1)\lambda_*
\mathbf{x}_*$. Hence, $-(A_{t_*}-(m-1)\lambda_* I_n)$ is a singular $M$-matrix.
The leading submatrix of \eqref{eq2.8b}, $A_{t_*}-\lambda_* I_n$, has at
least one positive real eigenvalue when $m>2$.

Next, we show that $\mathbf{0}\in \mathbb{R}^{n+1}$ is a regular value of $%
H_{\rm Z}:\mathbb{R}_{> 0}^{n+1}\times [0,1)\rightarrow \mathbb{R}^{n+1}$.

\begin{Theorem}
\label{thm2.6} Let $\mathcal{A}\in \mathbb{R}^{[m,n]}_{\geqslant 0}$ and $%
\mathcal{A}_0=\mathbf{x}_1\circ\cdots\circ\mathbf{x}_1$, where $\mathbf{x}%
_1\in \mathbb{R}^{n}_{>0}$ is generic. Then $\mathbf{0}\in \mathbb{R}^{n+1}$
is a regular value of the homotopy function $H_{\rm Z}:\mathbb{R}_{>
0}^{n+1}\times [0,1)\rightarrow \mathbb{R}^{n+1}$ in \eqref{eq2.7}.
\end{Theorem}

\begin{proof}
Let $P:\mathbb{R}^{n+1}_{>0}\times (0,1)\times \mathbb{R}^{n}_{>0}\rightarrow  \mathbb{R}^{n+1}$ be defined by $P(\mathbf{u},t,\mathbf{c})=H_{\rm Z}(\mathbf{u},t)$, where $\mathbf{u}=(\mathbf{x},\lambda)$ and $H_{\rm Z}(\mathbf{u},t)$ is given in \eqref{eq2.7} with $\mathcal{A}_0=\mathbf{c}\circ\cdots\circ\mathbf{c}\in \mathbb{R}^{[m,n]}_{>0}$. Now we show that $\mathbf{0}\in \mathbb{R}^{n+1}$ is a regular value of $P$. Let $(\mathbf{u}_*,t_*,\mathbf{c}_*)\in \mathbb{R}^{n+1}_{>0}\times (0,1)\times \mathbb{R}^{n}_{>0}$ be a solution of $P(\mathbf{u},t,\mathbf{c})=\mathbf{0}$, then
$\mathscr{D}_{\mathbf{u},t,\mathbf{c}}P(\mathbf{u}_*,t_*,\mathbf{c}_*)=[\mathscr{D}_{\mathbf{u},t}P(\mathbf{u}_*,t_*,\mathbf{c}_*)|\mathscr{D}_{\mathbf{c}}P(\mathbf{u}_*,t_*,\mathbf{c}_*)],
$
where $\mathscr{D}_{\mathbf{u},t}P(\mathbf{u}_*,t_*,\mathbf{c}_*)=\mathscr{D}_{\mathbf{x},\lambda,t}H_{\rm Z}(\mathbf{x}_*,\lambda_*,t_*)$ is given in \eqref{eq2.8} with $\mathbf{u}_*=(\mathbf{x}_*,\lambda_*)$ and $\mathcal{A}_0=\mathbf{c}_*\circ\cdots\circ\mathbf{c}_*$, and
\begin{align*}
\mathscr{D}_{\mathbf{c}}P(\mathbf{u}_*,t_*,\mathbf{c}_*)=(1-t_*)\left[\begin{array}{c}(\mathbf{c}_*^{\top}\mathbf{x}_*)^{m-1} I_n+(m-1)(\mathbf{c}_*^{\top}\mathbf{x}_*)^{m-2}\mathbf{c}_*\mathbf{x}_*^{\top}\\\hline \mathbf{0}^{\top}\end{array}\right]\in \mathbb{R}^{(n+1)\times n}.
\end{align*}
Since  $\mathbf{x}_*,\mathbf{c}_*\in \mathbb{R}^{n}_{>0}$, the matrix $(\mathbf{c}_*^{\top}\mathbf{x}_*)^{m-1} I_n+(m-1)(\mathbf{c}_*^{\top}\mathbf{x}_*)^{m-2}\mathbf{c}_*\mathbf{x}_*^{\top}$ is invertible. From \eqref{eq2.8}, the last row of the matrix $\mathscr{D}_{\mathbf{u},t}P(\mathbf{u}_*,t_*,\mathbf{c}_*)$ is nonzero, and hence rank$(\mathscr{D}_{\mathbf{u},t,\mathbf{c}}P(\mathbf{u}_*,t_*,\mathbf{c}_*))=n+1$. That is, $\mathbf{0}$ is a regular value of $P$. It follows from the Parameterized Sard's Theorem (Theorem \ref{thm2.6.5}) that for almost all $\mathbf{x}_1\in \mathbb{R}_{>0}^n$, $\mathbf{0}\in \mathbb{R}^{n+1}$ is a regular value of $H_{\rm Z}:\mathbb{R}_{> 0}^{n+1}\times (0,1)\rightarrow \mathbb{R}^{n+1}$ with $\mathcal{A}_0=\mathbf{x}_1\circ\cdots\circ\mathbf{x}_1$.

It remains to show that $\mathbf{0}$ is a regular value of $F^0_{\rm Z}(\mathbf{x},\lambda)\equiv H_{\rm Z}(\mathbf{x},\lambda,0)$ on $\mathbb{R}_{\geqslant 0}^{n+1}$. For almost all $\mathbf{x}_1\in \mathbb{R}_{>0}^n$, the system of polynomial equations $H_{\rm Z}(\mathbf{x},\lambda,0)=\mathbf{0}$ has only one solution $(\mathbf{x}_0,\lambda_0)$ in $\mathbb{R}_{>0}^{n+1}$, where $\mathbf{x}_0$ and  $\lambda_0=\|\mathbf{x}_1\|^m$ are given in \eqref{eqlamxz}. Now, we show that the Jacobian matrix
\begin{align*}
\mathscr{D}_{\mathbf{x},\lambda}H_{\rm Z}(\mathbf{x}_0,\lambda_0,0)=\left[\begin{array}{c|c}A_{0}-\lambda_0 I_n&-\mathbf{x}_0\\\hline
2\mathbf{x}_0^{\top}&0\end{array}\right]
\end{align*}
is invertible, where $A_0\equiv \mathscr{D}_{\mathbf{x}}(\mathcal{A}_0\mathbf{x}^{m-1})|_{\mathbf{x}=\mathbf{x}_0}=(m-1)\|\mathbf{x}_1\|^{m-2}\mathbf{x}_1\mathbf{x}_1^{\top}$ is given in \eqref{derAt}. It is easily seen that $A_0>0$ has only one nonzero eigenvalue $(m-1)\|\mathbf{x}_1\|^m$ corresponding eigenvector $\mathbf{x}_1$. Then $A_{0}-\lambda_0 I_n$ is nonsingular matrix when $m>2$.
\begin{itemize}
\item If  $m>2$, then the value $\mathbf{x}_0^{\top}(A_0-\lambda_0I_n)^{-1}\mathbf{x}_0=\frac{1}{(m-2)\|\mathbf{x}_0\|^m}\neq 0$, hence,  $\mathscr{D}_{\mathbf{x},\lambda}H_{\rm Z}(\mathbf{x}_0,\lambda_0,0)$ is invertible.
\item If $m=2$, then from Lemma \ref{lem3.3}, we obtain that  $\mathscr{D}_{\mathbf{x},\lambda}H_{\rm Z}(\mathbf{x}_0,\lambda_0,0)$ is invertible.
\end{itemize}
Since $\mathbf{0}$ is also a regular value of $H_{\rm Z}(\cdot,0)$ on $\mathbb{R}_{\geqslant 0}^{n+1}$,  $\mathbf{0}$ is a regular value of $H_{\rm Z}:\mathbb{R}_{> 0}^{n+1}\times [0,1)\rightarrow \mathbb{R}^{n+1}$.
\end{proof}

From Theorem \ref{thm2.6} and the implicit function theorem, we know that
the equation $H_{\rm Z}(\mathbf{x},\lambda, t)=\mathbf{0}$ has a solution curve $%
\mathbf{w}(s)$ with initial $\mathbf{w}(0)=(\mathbf{x}_0,\lambda_0,0)\equiv (%
\mathbf{x}_1/\|\mathbf{x}_1\|,\|\mathbf{x}_1\|^m,0)$,
\begin{align}  \label{eq2.9}
\mathbf{w}(s)\equiv (\mathbf{x}(s),\lambda(s),t(s))\in \mathbb{R}^{n}_{>
0}\times \mathbb{R}_{> 0}\times [0,1) \text{ for }s\in [0,s_{max}),
\end{align}
which can be parameterized by arc-length $s$, where $s_{max}$ is the largest
arc-length such that $\mathbf{w}(s)\in \mathbb{R}^{n}_{> 0}\times \mathbb{R}%
_{> 0}\times [0,1)$. Note that this curve, $\mathbf{w}(s)$ for $s\in
[0,s_{max})$, has no bifurcation and is bounded (by \eqref{equ_bound}). This
curve, $t(s)$ for $s\in[0,s_{max})$, may have turning points at some
parameters $s$. The following proposition shows that the turning point will
happen when $A_{t(s)}-\lambda(s)I$ is singular.

\begin{Proposition}
\label{prop2.0} Let $m>2$, $\mathcal{A}\in \mathbb{R}^{[m,n]}_{\geqslant 0}$ and $\mathcal{A}_0=\mathbf{x}_1\circ\cdots\circ%
\mathbf{x}_1\in \mathbb{R}^{[m,n]}_{>0}$, where $\mathbf{x}_1\in \mathbb{R}%
^{n}_{>0}$ is generic. Suppose that $\mathbf{w}(s)=(\mathbf{x}%
(s),\lambda(s),t(s))$ defined in \eqref{eq2.9} is the solution curve of %
\eqref{eq2.7}. If $t(s)$ has turning point at $s_*\in [0,s_{max})$ then $%
A_{t(s_*)}-\lambda(s_*)I$ is singular, where $A_{t(s_*)}\equiv \mathscr{D}_{%
\mathbf{x}}(\mathcal{A}(t(s_*))\mathbf{x}^{m-1})|_{\mathbf{x}=\mathbf{x}%
(s_*)}$.
\end{Proposition}

\begin{proof}
Suppose $A_{t(s_*)}-\lambda(s_*)I$ is invertible. Denote $(\mathbf{x}_*,\lambda_*,t_*)=(\mathbf{x}(s_*),\lambda(s_*),t(s_*))$. Since $(A_{t_*}-\lambda_*I)\mathbf{x}_*=(m-1)\mathcal{A}(t_*)\mathbf{x}_*^{m-1}-\lambda_*\mathbf{x}_*=(m-2)\lambda_*\mathbf{x}_*$, $m>2$  and $\lambda_*\neq 0$, we have $(A_{t_*}-\lambda_*I)^{-1}\mathbf{x}_*=\frac{1}{(m-2)\lambda_*}\mathbf{x}_*$.
Hence, $\alpha\equiv\mathbf{x}_*^{\top}(A_{t_*}-\lambda_*I)^{-1}\mathbf{x}_*=\frac{1}{(m-2)\lambda_*}>0$. Then the Jacobian matrix $\mathscr{D}_{\mathbf{x},\lambda}H_{\rm Z}(\mathbf{x}_*,\lambda_*,t_*)$ in \eqref{eq2.8b} is invertible. By the implicit function theorem, the solution curve $\mathbf{w}(s)$ can be parametrized by $t$ when $t$ approximates $t_{*}=t(s_*)$, a contradiction.
\end{proof}

\begin{Theorem}
\label{thm2.7} Let $\mathcal{A}\in \mathbb{R}^{[m,n]}_{\geqslant 0}$ and $\mathcal{A}_0=\mathbf{x}_1\circ\cdots\circ\mathbf{x}%
_1\in \mathbb{R}^{[m,n]}_{>0}$, where $\mathbf{x}_1\in \mathbb{R}^{n}_{>0}$
is generic. Suppose that $\mathbf{w}(s)=(\mathbf{x}(s),\lambda(s),t(s))$
 in \eqref{eq2.9} is the solution curve of \eqref{eq2.7}. Then $
\lim_{s\rightarrow s_{max}^{-}}t(s)=1$.
\end{Theorem}

\begin{proof}
Let $\mathbf{w}(s)\in \mathbb{R}^{n+1}_{> 0}\times [0,1)$ for $s\in[0,s_{max})$ be the solution curve of $H_{\rm Z}(\mathbf{x},\lambda, t)=\mathbf{0}$.  Suppose that the sequence $\{s_k\}_{k=1}^{\infty}\subset [0,s_{max})$ is increasing and $\lim_{k\rightarrow \infty}s_k=s_{max}$. Now, we show that $\lim_{k\rightarrow \infty}t(s_k)=1$.

Suppose that $\lim_{k\rightarrow \infty}t(s_k)\neq1$, then there exists a subsequence $\{s_{k_\ell}\}_{\ell=1}^{\infty}$ of $\{s_k\}_{k=1}^{\infty}$ such that $\lim_{\ell\rightarrow \infty}t(s_{k_{\ell}})=t_*\neq1$. It follows from \eqref{equ_bound} that the set $ (\mathbf{x}(s_{k_\ell}),\lambda(s_{k_\ell}))\in \mathbb{R}^{n}_{> 0}\times \mathbb{R}_{> 0}$ is bounded, then there is a subsequence $\{s_{k_{\hat{\ell}}}\}_{\hat{\ell}=1}^{\infty}$ of $\{s_{k_\ell}\}_{\ell=1}^{\infty}$ such that $\lim_{\hat{\ell}\rightarrow \infty}(\mathbf{x}(s_{k_{\hat{\ell}}}),\lambda(s_{k_{\hat{\ell}}}))=(\mathbf{x}_*,\lambda_*)\in  \mathbb{R}^{n}_{\geqslant 0}\times \mathbb{R}_{\geqslant 0}.$
It is easily seen that $H_{\rm Z}(\mathbf{x}_*,\lambda_*, t_*)=\mathbf{0}$, where $t_*\in [0,1)$. From Theorem \ref{thm2.1}, we have $\mathbf{x}_*>0$ and $\lambda_*>0$.
\begin{description}
  \item[Case1:] If $t_*\in (0,1)$, then the solution $(\mathbf{x}_*,\lambda_*, t_*)$ is in the set $ \mathbb{R}^{n}_{> 0}\times \mathbb{R}_{> 0}\times [0,1)$. It follows from Theorem \ref{thm2.6} that the equation $H_{\rm Z}(\mathbf{w})=\mathbf{0}$ has a solution curve in a certain neighborhood of $\mathbf{w}(s_{max})\equiv(\mathbf{x}_*,\lambda_*, t_*)$. This is a contradiction because $s_{max}$ is the largest arc-length such that $\mathbf{w}(s)\in \mathbb{R}^{n}_{> 0}\times \mathbb{R}_{> 0}\times [0,1)$.
  \item[Case2: ] If $t_*=0$, from Lemma \ref{lem2.2}, we obtain that $(\mathbf{x}_*,\lambda_*)=(\mathbf{x}_0,\lambda_0)$, where $(\mathbf{x}_0,\lambda_0)$ is defined in \eqref{eqlamxz}. It has been shown in Theorem \ref{thm2.6} that the Jacobian matrix $\mathscr{D}_{\mathbf{x},\lambda}H_{\rm Z}(\mathbf{x}_0,\lambda_0,0)$ defined in \eqref{eq2.8b} is invertible. By implicit function theorem, the solution curve $\mathbf{w}(s)$ in \eqref{eq2.9} can be parameterized by $t$ when $t$ approximates $0$ and there is no solution of $H_{\rm Z}(\mathbf{x},\lambda, t)=\mathbf{0}$ in $B_{\rho}((\mathbf{x}_0,\lambda_0, 0))$ other than $\mathbf{w}(s)$. This is contradiction.
\end{description}
Hence, $\lim_{s\rightarrow s_{max}^{-}}t(s)=1$.
\end{proof}

Theorem \ref{thm2.7} shows that $t(s)\rightarrow 1$ as $s\rightarrow
s_{max}^{-}$. Next, we will investigate the limit point of the curve, $(%
\mathbf{x}(s),\lambda(s))$, as $s\rightarrow s_{max}^{-}$, where $(\mathbf{x%
}(s),\lambda(s))$ is defined in \eqref{eq2.9}.

\begin{Theorem}
\label{thm2.8} Let  $\mathcal{A}\in \mathbb{R}^{[m,n]}_{\geqslant 0}$ and $\mathcal{A}_0=\mathbf{x}_1\circ\cdots\circ\mathbf{x}%
_1\in \mathbb{R}^{[m,n]}_{>0}$, where $\mathbf{x}_1\in \mathbb{R}^{n}_{>0}$
is generic. Then the solution curve $\mathbf{w}(s)=(\mathbf{x}%
(s),\lambda(s),t(s))$ for $s\in[0,s_{max})$ defined in \eqref{eq2.9}
satisfies the following properties.

\begin{itemize}
\item[(i)] There exist a sequence $\{s_k\}_{k=1}^{\infty}\subset [0,s_{max})$
and an accumulation point $\lambda_*\geqslant 0$ such that
$\lim_{k\rightarrow \infty}s_k=s_{max}\text{ and }\lim_{k\rightarrow
\infty}\lambda(s_k)=\lambda_*$;

\item[(ii)] For every such accumulation point $\lambda_*$, there exists a
vector $\mathbf{x}_*\geqslant 0$ such that the pair $(\lambda_*,\mathbf{x}_*)
$ is a Z-eigenpair of $\mathcal{A}$, i.e., $F_{\rm Z}(\mathbf{x}_*,\lambda_*)=%
\mathbf{0}$;

\item[(iii)] If $\mathcal{A}$ is weakly symmetric, then the eigenvalue curve
$\lambda(s)$ converges to $\lambda_*\geqslant 0$ as $s\rightarrow s_{max}^-$;

\item[(iv)] Let $(\mathbf{x}_*,\lambda_*)$ be such accumulation point of the
curve $(\mathbf{x}(s),\lambda(s))$ for $s\in [0,s_{\max})$. If $(\mathbf{x}%
_*,\lambda_*)$ is an isolated solution of $F_{\rm Z}(\mathbf{x},\lambda)=\mathbf{0}
$, then
\begin{align*}
\lim_{s\rightarrow s_{max}^-}(\mathbf{x}(s),\lambda(s))=(\mathbf{x}%
_*,\lambda_*).
\end{align*}
\end{itemize}
\end{Theorem}

\begin{proof}
$(i)$ Using the fact  that the set $\{\lambda(s)\ |\ s\in [0,s_{max})\}\subset\mathbb{R}_{>0}$ is bounded, the assertion $(i)$ can be obtained.

$(ii)$ Suppose that $\lambda(s_k)\rightarrow \lambda_*$. Since $\mathbf{x}(s_k)\in \{\mathbf{x}\in\mathbb{R}^{n}_{>0}\ |\ \|\mathbf{x}\|=1\}$ for each $k$, there is a subsequence $\{s_{k_\ell}\}_{\ell=1}^{\infty}$ of $\{s_k\}_{k=1}^{\infty}$ such that $\mathbf{x}(s_{k_\ell})\rightarrow \mathbf{x}_*\geqslant 0$ with $\|\mathbf{x}_*\|=1$ as $\ell\rightarrow \infty$.  Using the fact that $H_{\rm Z}(\mathbf{x}(s_{k_\ell}),\lambda(s_{k_\ell}),t(s_{k_\ell}))=\mathbf{0}$ and $\lim_{\ell\rightarrow \infty}t(s_{k_\ell})=1$ (see Theorem \ref{thm2.7}), we have $H_{\rm Z}(\mathbf{x}_*,\lambda_*,1)=F_{\rm Z}(\mathbf{x}_*,\lambda_*)=\mathbf{0}$. Hence,  the pair $(\lambda_*,\mathbf{x}_*)$ is a Z-eigenpair of $\mathcal{A}$.

$(iii)$ Suppose that $\lambda(s)$ for $s\in[0,s_{max})$ does not converge as $s\rightarrow s_{max}^-$. Then $\lambda(s)$  has two different  accumulation points, $\lambda_*^1$ and $\lambda_*^2$ (say $\lambda_*^1<\lambda_*^2$), as $s\rightarrow s_{max}^-$. Since the eigenvalue curve $\lambda(s)\in \mathbb{R}_{>0}$ is continuous  for $s\in[0,s_{max})$, we obtain that each point $\lambda_*\in [\lambda_*^1,\lambda_*^2]$, $\lambda_*$ is an accumulation point. From $(i)$ and $(ii)$, we obtain that the tensor $\mathcal{A}$ has infinitely many Z-eigenvalues.  This is a contradiction because $\mathcal{A}$ is weakly symmetric,  $\mathcal{A}$ has only finitely many Z-eigenvalues (see Proposition 3.10 in \cite{Chang-Pearson-Zhang:2013}). Hence, $\lim_{s\rightarrow s_{max}^-}\lambda(s)=\lambda_*\geqslant 0$.

$(iv)$ Suppose $(\mathbf{x}(s),\lambda(s))$ for $s\in[0,s_{max})$ has another accumulation point $(\hat{\mathbf{x}}_*,\hat{\lambda}_*)$ such that $\delta_*=\|(\mathbf{x}_*- \hat{\mathbf{x}}_*,\lambda_*-\hat{\lambda}_*)\|>0$. Then for each $\epsilon>0$, by continuity of $(\mathbf{x}(s),\lambda(s))$, there exists an increasing sequence $\{s_k\}\subset [0,s_{max})$ such that $\lim_{k\rightarrow \infty}s_k=s_{max}$ and
\begin{align*}
0<\min\{\delta_*/2,\epsilon/3\}\leqslant\|(\mathbf{x}(s_k)-\mathbf{x}_*,\lambda(s_k)-\lambda_*)\|\leqslant \epsilon/2 \text{ for each }k=1,2,\ldots.
\end{align*}
Since the sequence $(\mathbf{x}(s_k),\lambda(s_k))$ is bounded, there exists an accumulation point $(\tilde{\mathbf{x}}_*,\tilde{\lambda}_*)$ of  the sequence $\{(\mathbf{x}(s_k),\lambda(s_k))\}_{k=1}^{\infty}$ such that $0<\|(\tilde{\mathbf{x}}_*-\mathbf{x}_*,\tilde{\lambda}_*-\lambda_*)\|<\epsilon$. From $(i)$ and $(ii)$, $(\tilde{\mathbf{x}}_*,\tilde{\lambda}_*)$ is a solution of  $F_{\rm Z}(\mathbf{x},\lambda)=\mathbf{0}$. This is a contradiction because $(\mathbf{x}_*,\lambda_*)$ is an isolated solution of $F_{\rm Z}(\mathbf{x},\lambda)=\mathbf{0}$.
\end{proof}

Note that the condition of Theorem \ref{thm2.8} $(iv)$ holds generically. We conjecture that the convergence of solution curve, $(\mathbf{x}(s),\lambda(s))$ as $s\rightarrow s_{max}^-$, is guaranteed even without
this condition.

In the following, we show the degree of $F^0_{\rm Z}$ on $\mathbb{R}^{n+1}_{> 0}$ for $%
\mathbf{p}=\mathbf{0}$ is only dependent on the dimension $n$, where $F^0_{\rm Z}$
is defined in \eqref{eqZ}.

\begin{Lemma}
\label{lem3.9} Let $\mathcal{A}_0=\mathbf{x}_1\circ\cdots\circ\mathbf{x}%
_1\in \mathbb{R}^{[m,n]}_{>0}$, where $\mathbf{x}_1\in \mathbb{R}^{n}_{>0}$
is generic. Then $\mathrm{deg}(F^0_{\rm Z}, \mathbb{R}^{n+1}_{> 0},\mathbf{0}%
)=(-1)^{n-1}$, where $F^0_{\rm Z}$ is defined in \eqref{eqZ}.
\end{Lemma}

\begin{proof}
From Lemma \ref{lem2.2} $(i)$, we know that $(\mathbf{x}_0,\lambda_0)=(\mathbf{x}_1/\|\mathbf{x}_1\|,\|\mathbf{x}_1\|^{m})$ is the unique solution of  $F^0_{\rm Z}(\mathbf{x},\lambda)=\mathbf{0}$ on $\mathbb{R}^{n+1}_{\geqslant 0}$. Let $Q\in \mathbb{R}^{n\times n}$ be an orthogonal matrix such that $Q\mathbf{x}_0=\mathbf{e}_n\equiv [0,\cdots,0,1]^{\top}$ and $\widehat{Q}=\left[\begin{array}{c|c}Q&0\\\hline 0&1\end{array}\right]$. Then the Jacobian matrix
\begin{align*}
\mathscr{D}_{\mathbf{x},\lambda}F^0_{\rm Z}(\mathbf{x}_0,\lambda_0)&=\left[\begin{array}{c|c}(m-1)\|\mathbf{x}_1\|^{m-2}\mathbf{x}_1\mathbf{x}_1^{\top}-\lambda_0 I_n&-\mathbf{x}_0\\\hline
2\mathbf{x}_0^{\top}&0\end{array}\right]\\
&=\widehat{Q}^{\top}\left[\begin{array}{c|c}(m-1)\|\mathbf{x}_1\|^{m}\mathbf{e}_n\mathbf{e}_n^{\top}-\|\mathbf{x}_1\|^{m} I_n&-\mathbf{e}_n\\\hline
2\mathbf{e}_n^{\top}&0\end{array}\right]\widehat{Q}\\
&=\widehat{Q}^{\top} \left[(-\|\mathbf{x}_1\|^{m}I_{n-1})\oplus \left[\begin{array}{c|c}(m-2)\|\mathbf{x}_1\|^{m}&-1\\\hline 2&0\end{array}\right]\right]\widehat{Q}.
\end{align*}
So, ${\rm deg}(F^0_{\rm Z}, \mathbb{R}^{n+1}_{> 0},\mathbf{0})={\rm Sgn}({\rm det }(\mathscr{D}_{\mathbf{x},\lambda}F^0_{\rm Z}(\mathbf{x}_0,\lambda_0)))={\rm Sgn}\left((-1)^{n-1}2\|\mathbf{x}_1\|^{m(n-1)}\right)=(-1)^{n-1}$.
\end{proof}

\begin{Theorem}
\label{thm2.10} Suppose that $\mathcal{A}\in \mathbb{R}^{[m,n]}_{\geqslant 0}
$ is irreducible. Then $\mathrm{deg}(F_{\rm Z}, \mathbb{R}^{n+1}_{> 0},\mathbf{0})=(-1)^{n-1},$
where $F_{\rm Z}$ is defined in \eqref{eqZ}.
\end{Theorem}

\begin{proof}
Let  $\mathcal{A}_0=\mathbf{x}_1\circ\cdots\circ\mathbf{x}_1\in \mathbb{R}^{[m,n]}_{> 0}$ with $\mathbf{x}_1\in \mathbb{R}^{n}_{>0}$ being generic. Let $\mathbf{u}=(\mathbf{x}^{\top},\lambda)^{\top}\in \mathbb{R}^{n+1}$ and $H_{\rm Z}(\mathbf{u},t)$ be defined in \eqref{eq2.7}. Then $F^0_{\rm Z}(\mathbf{u})=H_{\rm Z}(\mathbf{u},0)$ and $F_{\rm Z}(\mathbf{u})=H_{\rm Z}(\mathbf{u},1)$. From \eqref{equ_bound}, there exists continuous function, $\rho(t)$ for $t\in [0,1]$, such that $(1-t)\mathcal{A}_0+t\mathcal{A}$ has no Z-eigenpair on the set $\mathbb{R}^{n+1}_{\geqslant 0}\backslash\Omega(t)$, where  $\Omega(t)\equiv \{\mathbf{u}\in \mathbb{R}^{n+1}_{> 0}\ |\ \|\mathbf{u}\|<\rho(t)\}$. Let $\rho =\max_{t\in[0,1]}\rho(t)$ and $\Omega\equiv \{\mathbf{u}\in \mathbb{R}^{n+1}_{> 0}\ |\ \|\mathbf{u}\|<\rho\}$ be a bounded open set. Then for all $t\in [0,1]$, $H_{\rm Z}(\mathbf{x}, \lambda,t)\neq \mathbf{0}$ on the set $\mathbb{R}^{n+1}_{\geqslant 0}\backslash\Omega$.
Since $\mathcal{A}$ is irreducible, $\mathcal{A}$ has no Z-eigenpair on $\partial(\mathbb{R}^{n+1}_{>0})$. From Theorem \ref{thm2.1}, we obtain that $H_{\rm Z}(\mathbf{u},t)\neq \mathbf{0}$ on $\partial\Omega\times [0,1]$. It follows from the Homotopy Invariance of Degree theorem (Theorem \ref{thm2.9}) and Lemma \ref{lem3.9} that
${\rm deg}(F_{\rm Z}, \mathbb{R}^{n+1}_{> 0},\mathbf{0})={\rm deg}(F_{\rm Z}, \Omega,\mathbf{0})={\rm deg}(F^0_{\rm Z}, \Omega,\mathbf{0})
={\rm deg}(F^0_{\rm Z}, \mathbb{R}^{n+1}_{> 0},\mathbf{0})=(-1)^{n-1}. $
\end{proof}

The following result can be obtained from Theorem \ref{thm2.10} directly.

\begin{Corollary}
\label{cor2.11} Let $\mathcal{A}\in \mathbb{R}^{[m,n]}_{\geqslant 0}$ be
irreducible. Suppose that all solutions of $F_{\rm Z}(\mathbf{x},\lambda)=\mathbf{0%
}$ in $\mathbb{R}^{n+1}_{\geqslant 0}$ are isolated, where $F_{\rm Z}$ is defined
in \eqref{eqZ}. Then the number of positive Z-eigenpairs of $\mathcal{A}$,
counting multiplicities, is $2k+1$ for some integer $k\geqslant 0$.
\end{Corollary}

\subsection{Computing the H-eigenpair of nonnegative tensors}

Given a tensor $\mathcal{A}\in \mathbb{R}_{\geqslant 0}^{[m,n]}$, let $%
\mathcal{A}_0\in \mathbb{R}_{> 0}^{[m,n]}$ in \eqref{eq2.1} be a rank-1
positive tensor. Let
\begin{align}  \label{eqlamx}
\lambda_0=\prod_{k=2}^m\left(\mathbf{x}_k^{\top}\mathbf{x}%
_1^{[1/(m-1)]}\right)>0,\ \ \ \mathbf{x}_0=\frac{\mathbf{x}_1^{[1/(m-1)]}}{\|%
\mathbf{x}_1^{[1/(m-1)]}\|}>0.
\end{align}
It follows from Lemma \ref{lem2.2} $(ii)$ that $(\lambda_0,\mathbf{x}_0)\in \mathbb{%
R}_{>0}^{n+1}$ is a H-eigenpair of $\mathcal{A}_0$ and $\mathbf{x}_0$ is
the unique H-eigenvector of $\mathcal{A}_0$ in $\mathbb{R}^{n}_{\geqslant
0}$. Then $F^0_{\rm H}(\mathbf{x}_0,\lambda_0)=H_{\rm H}(\mathbf{x}_0,\lambda_0,0)=\mathbf{%
0}$, where $F^0_{\rm H}$ and $H_{\rm H}$ are defined in \eqref{eqH} and \eqref{eq2.2},
respectively.

\begin{Lemma}
\label{lem2.3} Suppose that $(\mathbf{x}_*,\lambda_*,t_*)\in \mathbb{R}%
^{n+1}_{> 0}\times [0,1)$ is a solution of $H_{\rm H}(\mathbf{x},\lambda,t)=\mathbf{%
0}$. Then the Jacobian
matrix $\mathscr{D}_{\mathbf{x},\lambda}H_{\rm H}(\mathbf{x}_*,\lambda_*,t_*)\in
\mathbb{R}^{(n+1)\times (n+1)}$ is invertible.
\end{Lemma}

\begin{proof}
Suppose that $(\mathbf{x}_*,\lambda_*,t_*)\in \mathbb{R}^{n+1}_{> 0}\times [0,1)$ is a solution of $H_{\rm H}(\mathbf{x},\lambda,t)=\mathbf{%
0}$, we have $\mathcal{A}(t_*)\mathbf{x}_*^{m-1}=\lambda_*\mathbf{x}_*^{[m-1]}$, where $\mathcal{A}(t_*)\in \mathbb{R}_{>0}^{[m,n]}$ is defined in \eqref{eqAt}.  Then
\begin{align*}
\mathscr{D}_{\mathbf{x},\lambda}H_{\rm H}(\mathbf{x}_*,\lambda_*,t_*)=\left[\begin{array}{c|c}A_{t_*}-(m-1)\lambda_* [\![\mathbf{x}_*^{[m-2]}]\!]&-\mathbf{x}_*^{[m-1]}\\\hline
2\mathbf{x}_*^{\top}&0\end{array}\right]\in \mathbb{R}^{(n+1)\times (n+1)},
\end{align*}
where  $[\![\mathbf{x}]\!]$ denotes a squared diagonal matrix with the elements of vector $\mathbf{x}$ on the main diagonal, and $A_{t_*}=\mathscr{D}_{\mathbf{x}}(\mathcal{A}(t_*)\mathbf{x}^{m-1})|_{\mathbf{x}=\mathbf{x}_*}$ is given in \eqref{derAt}. Using the fact that $\mathcal{A}(t_*)>0$ and $\mathbf{x}_*>0$, we obtain $A_{t_*}>0$. Since $A_{t_*}\mathbf{x}_*=(m-1)\mathcal{A}(t_*)\mathbf{x}_*^{m-1}=(m-1)\lambda_*\mathbf{x}_*^{[m-1]}$, the matrix $-(A_{t_*}-(m-1)\lambda_* [\!|\mathbf{x}_*^{[m-2]}|\!])$ is singular $M$-matrix with singular vector $\mathbf{x}_*$.   It follows from Lemma \ref{lem3.3} that $\mathscr{D}_{\mathbf{x},\lambda}H_{\rm H}(\mathbf{x}_*,\lambda_*,t_*)$ is invertible.
\end{proof}

Since $\left(\mathbf{x}_0,\lambda_0,0\right)$ is a solution of $H_{\rm H}(\mathbf{x%
},\lambda, t)=\mathbf{0}$, where $\lambda_0$ and $\mathbf{x}_0$ are given in %
\eqref{eqlamx}. By Lemma \ref{lem2.3} and the implicit function theorem, we
know that the equation $H_{\rm H}(\mathbf{x},\lambda, t)=\mathbf{0}$ has a
solution curve with initial $\left(\mathbf{x}_0,\lambda_0,0\right)$, $%
\mathbf{w}(t)=(\mathbf{x}(t),\lambda(t),t)$ for $t\in [0,1)$ and $\mathbf{w}%
(0)=\left(\mathbf{x}_0,\lambda_0,0\right)$, which can be parameterized by $t$%
. Let
\begin{align}  \label{eq2.4}
C_{(\mathbf{x}_0,\lambda_0)}=\{(\mathbf{x}(t),\lambda(t),t)|t\in[0,1)\},
\text{ with }(\mathbf{x}(0),\lambda(0),0)=\left(\mathbf{x}%
_0,\lambda_0,0\right)
\end{align}
be the set of solution curve $\mathbf{w}(t)$ for $t\in[0,1)$. The following
theorem shows that $C_{(\mathbf{x}_0,\lambda_0)}$ is the solution set of $%
H_{\rm H}(\mathbf{x},\lambda, t)=\mathbf{0}$ in $\mathbb{R}^{n+1}_{\geqslant
0}\times [0,1)$.

\begin{Theorem}
\label{thm2.2} $C_{(\mathbf{x}_0,\lambda_0)}$ is the solution set of $H_{\rm H}(%
\mathbf{x},\lambda, t)=\mathbf{0}$ in $\mathbb{R}^{n+1}_{\geqslant 0}\times
[0,1)$.
\end{Theorem}

\begin{proof}
Suppose that $(\mathbf{x}_*,\lambda_*,t_*)\in  \mathbb{R}^{n+1}_{\geqslant 0}\times [0,1)$ is a solution of $H_{\rm H}(\mathbf{x},\lambda, t)=\mathbf{0}$ and $(\mathbf{x}_*,\lambda_*,t_*)\neq (\mathbf{x}(t_*),\lambda(t_*),t_*)\in C_{(\mathbf{x}_0,\lambda_0)}$. By Lemma \ref{lem2.3} and the implicit function theorem, there is a solution curve,  $\mathbf{w}_*(t)=(\mathbf{x}_*(t),\lambda_*(t),t)$ for $t\in [0,t_*]$ with $\mathbf{w}_*(t_*)=(\mathbf{x}_*,\lambda_*,t_*)$, which is parameterized by $t$. Let
$C_{(\mathbf{x}_*,\lambda_*)}=\{(\mathbf{x}_*(t),\lambda_*(t),t)|t\in[0,t_*]\}$ be the set of the solution curve.
Then from Lemma \ref{lem2.3}, we have $C_{(\mathbf{x}_*,\lambda_*)}\bigcap C_{(\mathbf{x}_0,\lambda_0)}=\emptyset$. From Theorem \ref{thm2.1}, we obtain that $C_{(\mathbf{x}_*,\lambda_*)}\subseteq \mathbb{R}^{n+1}_{\geqslant 0}\times [0,1)$. Since $(\mathbf{x}_*(0),\lambda_*(0))$ satisfies $F^0_{\rm H}(\mathbf{x},\lambda)=H_{\rm H}(\mathbf{x},\lambda, 0)=\mathbf{0}$ and hence, $(\lambda_*(0),\mathbf{x}_*(0))\in \mathbb{R}^{n+1}_{\geqslant 0}$ is a H-eigenpair of  the rank-1 tensor $\mathcal{A}_0$. It follows from Lemma \ref{lem2.2} $(ii)$ that $(\lambda_*(0),\mathbf{x}_*(0))=(\lambda_0,\mathbf{x}_0)$,  where $\lambda_0$ and $\mathbf{x}_0$ are defined in \eqref{eqlamx}. This is a contradiction because $C_{(\mathbf{x}_*,\lambda_*)}\bigcap C_{(\mathbf{x}_0,\lambda_0)}=\emptyset$.
\end{proof}

Suppose that $(\mathbf{x}(t),\lambda(t),t)$ for $t\in [0,1)$ is the solution
curve of $H_{\rm H}(\mathbf{x},\lambda, t)=\mathbf{0}$. Since $\|\mathbf{x}(t)\|=1$
and $\mathcal{A}(t)$ in \eqref{eqAt} is bounded for $t\in [0,1)$, then $%
\|(\lambda(t),\mathbf{x}(t))\|$ is bounded for $t\in [0,1)$. Suppose that $(%
\mathbf{x}_*,\lambda_*,1)$ is an accumulation point of $C_{(\mathbf{x}%
_0,\lambda_0)}$. Then $\mathbf{x}_*\in \mathbb{R}^{n}_{\geqslant 0}$ with $\|%
\mathbf{x}_*\|=1$, $\lambda_*\geqslant 0$ and $F_{\rm H}(\mathbf{x}%
_*,\lambda_*)=H_{\rm H}(\mathbf{x}_*,\lambda_*,1)=\mathbf{0}$. Hence, $(\lambda_*,%
\mathbf{x}_*)$ is a H-eigenpair of $\mathcal{A}$. In \cite%
{Hu-Huang-Ling-Qi:2013}, the authors shown that the eigenvalues of tensor $%
\mathcal{A}$ are the roots of a nonzero polynomial, hence, the eigenvalues
of $\mathcal{A}$ are isolated and we have $\lim_{t\rightarrow
1^-}\lambda(t)=\lambda_*$. Let
\begin{align*}
\Gamma_{\lambda_*}=\{\mathbf{x}_*\in \mathbb{R}^{n}_{\geqslant 0}| (\mathbf{x%
}_*,\lambda_*,1)\text{ is an accumulation point of }C_{(\mathbf{x}%
_0,\lambda_0)}\}.
\end{align*}
Then $\Gamma_{\lambda_*}\neq \emptyset$ is connected. It is easily seen that
for each $\mathbf{x}_*\in \Gamma_{\lambda_*}$, $(\lambda_*,\mathbf{x}_*)$ is
a H-eigenpair of $\mathcal{A}$. The following theorem can be obtained
directly.

\begin{Theorem}
\label{thm2.3} Let $\mathbf{x}_*\in \Gamma_{\lambda_*}$. If $(\mathbf{x}%
_*,\lambda_*)$ is an isolated solution of $F_{\rm H}(\mathbf{x},\lambda)=\mathbf{0}
$ then $\Gamma_{\lambda_*}=\{\mathbf{x}_*\}$ and
$\lim_{t\rightarrow 1^-}(\mathbf{x}(t),\lambda(t))=(\mathbf{x}_*,\lambda_*),$
where $(\mathbf{x}(t),\lambda(t),t)$ for $t\in[0,1)$ is the solution curve
of $H_{\rm H}(\mathbf{x},\lambda, t)=\mathbf{0}$.
\end{Theorem}


\begin{Theorem}
\label{thm2.4} Let $\mathcal{A}\geqslant 0$ be weakly irreducible, then the
nonnegative solution of $F_{\rm H}(\mathbf{x},\lambda)=\mathbf{0}$ is isolated and
hence $\lim_{t\rightarrow 1^-}(\mathbf{x}(t),\lambda(t))=(\mathbf{x}%
_*,\lambda_*)$.
\end{Theorem}

\begin{proof}
Let $\mathbf{x}_*\in \Gamma_{\lambda_*}$. From Theorem \ref{thm2.3}, it suffices to show that $(\mathbf{x}_*,\lambda_*)$ is an isolated solution of $F_{\rm H}(\mathbf{x},\lambda)=\mathbf{0}$. The Jacobian matrix of $F_{\rm H}$ at $(\mathbf{x}_*,\lambda_*)$ is
\begin{align*}
\mathscr{D}_{\mathbf{x},\lambda}F_{\rm H}(\mathbf{x}_*,\lambda_*)=\left[\begin{array}{c|c}A_1-(m-1)\lambda_* [\![\mathbf{x}_*^{[m-2]}]\!]&-\mathbf{x}_*^{[m-1]}\\\hline
2\mathbf{x}_*^{\top}&0\end{array}\right]\in \mathbb{R}^{(n+1)\times (n+1)},
\end{align*}
where $A_1=\mathscr{D}_{\mathbf{x}}(\mathcal{A}\mathbf{x}^{m-1})|_{\mathbf{x}=\mathbf{x}_*} \in \mathbb{R}^{n\times n}_{\geqslant 0}$ is given in \eqref{derAt}. Since $\mathcal{A}$ is weakly irreducible, $\lambda_*>0$ and $\mathbf{x}_*>0$. We next show that $A_1$ is invertible before proving that $\mathscr{D}_{\mathbf{x},\lambda}F_{\rm H}(\mathbf{x}_*,\lambda_*)$ is invertible.

Suppose that $A_1=[A^1_{i,j}]$ is reducible, then there exists a nonempty proper subset $S\subset \{1,2,\cdots,n\}$ such that
 $A^1_{i,j}=0,\ \forall \ i\in S,\ \forall \ j \notin  S,$
 which implies
 \begin{eqnarray}\label{derAt2}
0 &=&A_{i,j}^{1}=\left( \sum_{k=2}^{m}\mathcal{A}\times _{2}\mathbf{x}_{\ast
}\cdots \times _{k-1}\mathbf{x}_{\ast }\times _{k+1}\mathbf{x}_{\ast }\cdots
\times _{m}\mathbf{x}_{\ast }\right) _{i,j} \nonumber  \\
&\geq&\left( \mathcal{A}\times _{2}\mathbf{x}_{\ast }\cdots \times _{j-1}%
\mathbf{x}_{\ast }\times _{j+1}\mathbf{x}_{\ast }\cdots \times _{m}\mathbf{x}%
_{\ast }\right) _{i,j}.
\end{eqnarray}
Since $\mathcal{A}$ is weakly irreducible, there exist $i\in S$ and $i_2,\ldots, i_m$ with at least one $j=i_q \notin S$ such that $\mathcal{A}_{i,i_2,\ldots,i_{q-1},j,i_{q+1}\ldots,i_m}\neq 0$. It follows from $\mathcal{A}\geqslant 0$, $\mathbf{x}_*>0$ and \eqref{derAt2} that
$(\mathcal{A}\times_2\mathbf{x}_*\cdots \times_{i_q-1}\mathbf{x}_*\times_{i_q+1}\mathbf{x}_*\cdots \times_m\mathbf{x}_*)_{i,j}> 0$
 and hence $A^1_{i,j}>0$, where $i\in S$ and $j \notin S$. This is a contradiction. So, $A_1\geqslant 0$ is irreducible.

 Using the fact that $(\lambda_*,\mathbf{x}_*)\in \mathbb{R}^{n+1}_{>0}$ is H-eigenpair of $\mathcal{A}$, we obtain that $(A_1-(m-1)\lambda_* [\!|\mathbf{x}_*^{[m-2]}|\!])\mathbf{x}_*=\mathbf{0}$, i.e., $-(A_1-(m-1)\lambda_* [\!|\mathbf{x}_*^{[m-2]}|\!])$ is irreducible singular $M$-matrix. It follows from Lemma \ref{lem3.3} that the Jacobian matrix $\mathscr{D}_{\mathbf{x},\lambda}F_{\rm H}(\mathbf{x}_*,\lambda_*)$ is invertible.
\end{proof}

Theorem \ref{thm2.4} shows that for each weakly irreducible nonnegative tensor $\mathcal{A}$, we can compute the unique positive H-eigenpair $(\lambda_*,
\mathbf{x}_*)$ of $\mathcal{A}$ by tracing the solution curve of $H_{\rm H}(%
\mathbf{x}, \lambda,t)=\mathbf{0}$ in \eqref{eq2.2} with initial $\left(\mathbf{x}_0,\lambda_0,0\right)$, where $\lambda_0$
and $\mathbf{x}_0$ are defined in \eqref{eqlamx}.

\section{Algorithms}

A continuation method usually follows the solution curves of $H(\mathbf{u}%
,t)=\mathbf{0}$ with prediction and correction steps, where $H:\mathbb{R}%
^{n+1}\rightarrow \mathbb{R}^n$ is a continuously differentiable function.
In this section, we propose homotopy continuation methods to compute the
nonnegative Z-eigenpair and H-eigenpair of a tensor $\mathcal{A}\in
\mathbb{R}_{\geqslant 0}^{[m,n]}$. In addition, if $\mathcal{A}$ is
irreducible, a novel continuation method is proposed to compute an odd number
of positive Z-eigenpairs of $\mathcal{A}$.


\subsection{Pseudo-arclength continuation method for computing an odd number of $%
Z$-eigenpairs of irreducible nonnegative tensors}

Given a nonnegative tensor $\mathcal{A}\in \mathbb{R}^{[m,n]}_{\geqslant 0}$%
. Let $\mathcal{A}_0=\mathbf{x}_1\circ\cdots\circ\mathbf{x}_1\in \mathbb{R}%
^{[m,n]}_{> 0}$, where $\mathbf{x}_1\in \mathbb{R}^n_{>0}$ is generic. Then $%
(\lambda_0,\mathbf{x}_0)=(\|\mathbf{x}_1\|^m,\frac{\mathbf{x}_1}{\|\mathbf{x}%
_1\|})$ is the positive Z-eigenpair of $\mathcal{A}_0$. Theorems \ref%
{thm2.6} and \ref{thm2.7} show that the solution curve of $H_{\rm Z}(%
\mathbf{x},\lambda,t)=\mathbf{0}$, $\mathbf{w}(s)=(\mathbf{x}%
(s),\lambda(s),t(s))$ for $s\in [0,s_{max})$, has no bifurcation and $%
t(s)\rightarrow 1^{-}$ as $s\rightarrow s_{max}^{-}$, where the function $%
H_{\rm Z}$ is defined in \eqref{eq2.7}.
This solution curve  may have turning points at some parameters $s\in [0,s_{max})$, it is natural to employ pseudo-arclength continuation method (see \cite{Keller:1987}) for
tracking the solution curve of $H_{\rm Z}(\mathbf{x},\lambda,t)=\mathbf{0}$ with
initial $(\mathbf{x}_0,\lambda_0,0)$. Denote $\mathbf{w}=(\mathbf{x}%
^{\top},\lambda,t)^{\top}\in \mathbb{R}^{n+2}$. Then $H_{\rm Z}(\mathbf{w})=H_{\rm Z}(%
\mathbf{x},\lambda,t)=\mathbf{0}$. Theorem \ref{thm2.8} $(iv)$ shows that if
the roots of $F_{\rm Z}(\mathbf{x},\lambda)=H_{\rm Z}(\mathbf{x},\lambda,1)=\mathbf{0}$
in $\mathbb{R}_{\geqslant 0}^{n+1}$ are isolated, then we can compute a
nonnegative Z-eigenpair of $\mathcal{A}$ by tracking the solution curve
with initial $\mathbf{w}_0=(\mathbf{x}_0,\lambda_0,0)$. To follow the
solution curve, we use the prediction-correction process. The prediction and
correction steps are described as follows.

{\leftmargini=7mm
\leftmarginii=6mm
\begin{itemize}
\item Prediction step: Suppose that $\mathbf{w}_{i}\in \mathbb{R}%
^{n+1}_{\geqslant 0}$ is a point lying (approximately) on a solution curve
of $H_{\rm Z}(\mathbf{w})=\mathbf{0}$. The Euler predictor
\begin{align*}
\mathbf{w}_{i+1,1}=\mathbf{w}_{i}+\Delta s_i \dot{\mathbf{w}}_{i}
\end{align*}
is used to predict a new point. Here, $\dot{\mathbf{w}}_{i}\in \mathbb{R}%
^{n+2}$ is the unit tangent vector of the solution curve of $H_{\rm Z}(\mathbf{w})=%
\mathbf{0}$ at $\mathbf{w}_{i}$ and $\Delta s_i >0$ is a suitable step
length. Let $\mathscr{D}_{\mathbf{w}}H_{\rm Z}(\mathbf{w}_i)\in \mathbb{R}%
^{(n+1)\times (n+2)}$ in \eqref{eq2.8a} be the Jacobian matrix of $H_{\rm Z}$ at $%
\mathbf{w}=\mathbf{w}_i$. The unit tangent vector $\dot{\mathbf{w}}_{i}$
should satisfy the linear system $\mathscr{D}_{\mathbf{w}}H_{\rm Z}(\mathbf{w}_i)%
\dot{\mathbf{w}}_{i}=\mathbf{0}$ and $\dot{\mathbf{w}}_{i}^{\top}\dot{%
\mathbf{w}}_{i-1}>0$ if $i\geqslant 1$. If $i=0$, we choose the unit tangent
vector $\dot{\mathbf{w}}_{0}$ such that the last component of $\dot{\mathbf{w%
}}_{0}$ is positive.

\item Correction step: Let $c_i=\dot{\mathbf{w}}_{i}^{\top}\mathbf{w}_{i+1,1}
$ be a constant. We use Newton's method to compute the approximate solution
of system
\begin{align*}
\left\{%
\begin{array}{l}
H_{\rm Z}(\mathbf{w})=\mathbf{0}, \\
\dot{\mathbf{w}}_{i}^{\top}\mathbf{w}-c_i=0,%
\end{array}
\right.
\end{align*}
with initial value $\mathbf{w}_{i+1,1}$. The iteration $\mathbf{w}%
_{i+1,\ell+1}=\mathbf{w}_{i+1,\ell}+\delta_{\ell}$ is computed for $%
\ell=1,2,\ldots$, where $\delta_{\ell}$ satisfies the linear system
\begin{align*}
\left[%
\begin{array}{c}
\mathscr{D}_{\mathbf{w}}H_{\rm Z}(\mathbf{w}_{i+1,\ell}) \\
\dot{\mathbf{w}}_{i}^{\top}%
\end{array}%
\right] \delta_{\ell}=-\left[%
\begin{array}{c}
H_{\rm Z}(\mathbf{w}_{i+1,\ell}) \\
\dot{\mathbf{w}}_{i}^{\top}\mathbf{w}_{i+1,\ell}-c_i%
\end{array}%
\right].
\end{align*}
If $\{\mathbf{w}_{i+1,\ell}\}$ converges until $\ell=\ell_{\infty}$, then we
accept $\mathbf{w}_{i+1}=\mathbf{w}_{i+1,\ell_{\infty}}$ as a new
approximation to the solution curve of $H_{\rm Z}(\mathbf{w})=\mathbf{0}$.
\end{itemize}}

Suppose that $\mathcal{A}\in \mathbb{R}^{[m,n]}_{\geqslant 0}$ is
irreducible and all solutions of $F_{\rm Z}(\mathbf{x},\lambda)=\mathbf{0}$ in $%
\mathbb{R}^{n+1}_{\geqslant 0}$ are isolated. Then Corollary \ref{cor2.11}
shows that the number of positive Z-eigenpairs of $\mathcal{A}$, counting
multiplicities, is odd. In the following, we propose a novel algorithm for
computing an odd number of positive Z-eigenpairs. The following theorem is
useful to construct the algorithm.

\begin{Theorem}
\label{thm2.11} Let $\mathcal{A}\in \mathbb{R}^{[m,n]}_{\geqslant 0}$ be
irreducible. Suppose that $\mathbf{0}$ is a regular value of $F_{\rm Z}:\mathbb{R}%
_{\geqslant 0}^{n+1}\rightarrow \mathbb{R}^{n+1}$ which is defined in %
\eqref{eqZ}. Let $\mathbf{x}_1,\mathbf{x}_2\in \mathbb{R}^n_{>0}$ be generic
and $H_{\rm Z,1}(\mathbf{x},\lambda,t)$ and $H_{\rm Z,2}(\mathbf{x},\lambda,t)$ be the
homotopy functions constructed in \eqref{eq2.7} with $\mathcal{A}_{0,1}=%
\mathbf{x}_1\circ\cdots\circ\mathbf{x}_1$ and $\mathcal{A}_{0,2}=\mathbf{x}%
_2\circ\cdots\circ\mathbf{x}_2$, respectively. Assume that $(\mathbf{x}_{*,1},\lambda_{*,1})$ and $(\mathbf{x}_{*,2},\lambda_{*,2})$ are accumulation
points of solution curves of $H_{\rm Z,1}(\mathbf{x},\lambda,t)=\mathbf{0}$ and $%
H_{\rm Z,2}(\mathbf{x},\lambda,t)=\mathbf{0}$, respectively. If $(\mathbf{x}%
_{*,1},\lambda_{*,1})\neq(\mathbf{x}_{*,2},\lambda_{*,2})$, then

\begin{itemize}
\item[(i)] $\mathrm{Sgn}(\mathrm{det }(\mathscr{D}_{\mathbf{x},\lambda}F_{\rm Z}(%
\mathbf{x}_{*,1},\lambda_{*,1})))=\mathrm{Sgn}(\mathrm{det }(\mathscr{D}_{%
\mathbf{x},\lambda}F_{\rm Z}(\mathbf{x}_{*,2},\lambda_{*,2})))=(-1)^{n-1}$;

\item[(ii)] there exists a smooth solution curve of $H_{\rm Z,1}(\mathbf{x}%
,\lambda,t)=\mathbf{0}$,
\begin{align*}
\hat{\mathbf{w}}(s)\equiv (\hat{\mathbf{x}}(s),\hat{\lambda}(s),\hat{t}%
(s))\in \mathbb{R}_{> 0}^{n+1}\times[0,1]\ \text{ for }\ s\in [0,\hat{s%
}_{max})
\end{align*}
with initial $\hat{\mathbf{w}}(0)=(\mathbf{x}_{*,2},\lambda_{*,2},1)$, where $\hat{s}_{max}$ is the largest arc-length such that $\hat{\mathbf{w}}%
(s)\in\mathbb{R}_{> 0}^{n+1}\times[0,1)$;

\item[(iii)] $\lim_{s\rightarrow \hat{s}_{max}^-}\hat{t}(s)=1$;

\item[(iv)] $\lim_{s\rightarrow \hat{s}_{max}^-}(\hat{\mathbf{x}}(s),\hat{%
\lambda}(s))=(\hat{\mathbf{x}}_*,\hat{\lambda}_*)\in \mathbb{R}_{> 0}^{n+1}$%
. Then $(\hat{\lambda}_*,\hat{\mathbf{x}}_*)$ is a Z-eigenpair of $%
\mathcal{A}$ and $(\hat{\mathbf{x}}_*,\hat{\lambda}_*)\neq (\mathbf{x}%
_*^i,\lambda_*^i)$, for $i=1,2$. In fact, $\mathrm{Sgn}(\mathrm{det }(%
\mathscr{D}_{\mathbf{x},\lambda}F_{\rm Z}(\hat{\mathbf{x}}_*,\hat{\lambda}%
_*)))=(-1)^n$.
\end{itemize}
\end{Theorem}

\begin{proof}
$(i)$ From the definitions of $H_{\rm Z,1}$ and $H_{\rm Z,2}$, we have $H_{\rm Z,1}(\mathbf{x},\lambda,0)=F^0_{\rm Z, 1}(\mathbf{x},\lambda)$, $H_{\rm Z,2}(\mathbf{x},\lambda,0)=F^0_{\rm Z, 2}(\mathbf{x},\lambda)$ and $H_{\rm Z,1}(\mathbf{x},\lambda,1)=H_{\rm Z,2}(\mathbf{x},\lambda,1)=F_{\rm Z}(\mathbf{x},\lambda)$, where $F^0_{\rm Z, 1}$ and $F^0_{\rm Z, 2}$ are of the form in \eqref{eqZ} with $\mathcal{A}_0=\mathcal{A}_{0,1}$ and $\mathcal{A}_0=\mathcal{A}_{0,2}$, respectively. For $i=1,2$, let $\mathbf{w}_i(s)\equiv (\mathbf{x}_i(s),\lambda_i(s),t_i(s))\in \mathbb{R}_{> 0}^{n+1}\times[0,1)$ for $s\in [0,s_{i,max})$ be the solution curve of $H_{{\rm Z},i}(\mathbf{x},\lambda,t)=\mathbf{0}$ with initial $\mathbf{w}_i(0)=(\mathbf{x}_i/\|\mathbf{x}_i\|, \|\mathbf{x}_i\|^m,0)$. Here $s_{i,max}$  is the largest arc-length such that $\mathbf{w}_i(s)\in\mathbb{R}_{> 0}^{n+1}\times[0,1)$.
Since $\mathbf{0}$ is a regular value of  $F_{\rm Z}$, from Theorem \ref{thm2.8} (\emph{iv}), we have
$\lim_{s\rightarrow s_{1,max}^-}\mathbf{w}_1(s)=(\mathbf{x}_{*,1},\lambda_{*,1},1)$ and $\lim_{s\rightarrow s_{2,max}^-}\mathbf{w}_2(s)=(\mathbf{x}_{*,2},\lambda_{*,2},1).$
By the Homotopy Invariance of Degree theorem (Theorem \ref{thm2.9}) and Lemma \ref{lem3.9}, we obtain that
\begin{align*}
&{\rm Sgn}({\rm det }(\mathscr{D}_{\mathbf{x},\lambda}F_{\rm Z}(\mathbf{x}_{*,1},\lambda_{*,1})))={\rm Sgn}({\rm det }(\mathscr{D}_{\mathbf{x},\lambda}F^0_{\rm Z, 1}(\mathbf{x}_1/\|\mathbf{x}_1\|, \|\mathbf{x}_1\|^m)))=(-1)^{n-1}, \text{ and }\\
&{\rm Sgn}({\rm det }(\mathscr{D}_{\mathbf{x},\lambda}F_{\rm Z}(\mathbf{x}_{*,2},\lambda_{*,2})))={\rm Sgn}({\rm det }(\mathscr{D}_{\mathbf{x},\lambda}F^0_{\rm Z, 2}(\mathbf{x}_2/\|\mathbf{x}_2\|, \|\mathbf{x}_2\|^m)))=(-1)^{n-1}.
\end{align*}

$(ii)$ Since $\mathbf{0}$ is a regular value of $F_{\rm Z}$ and $(\mathbf{x}_{*,2},\lambda_{*,2})$ is a solution of $F_{\rm Z}(\mathbf{x},\lambda)=\mathbf{0}$, $\mathscr{D}_{\mathbf{x},\lambda}F_{\rm Z}(\mathbf{x}_{*,2},\lambda_{*,2})$ is invertible. From Theorem \ref{thm2.6}, this assertion can be obtained.

$(iii)$ Since the equation $H_{\rm Z,1}(\mathbf{x},\lambda,0)=F^0_{\rm Z, 1}(\mathbf{x},\lambda)=\mathbf{0}$ has only one solution $(\mathbf{x}_0,\lambda_0)=(\mathbf{x}_1/\|\mathbf{x}_1\|,\|\mathbf{x}_1\|^m)$ in $\mathbb{R}^{n+1}_{>0}$ and $(\mathbf{x}_{*,1},\lambda_{*,1},1)$ is the accumulation point of the set $\{\mathbf{w}_1(s)\ |\ s\in [0,s_{1,max}) \}$, we obtain that $\hat{t}(s)$ does not converge to $0$ as $s\rightarrow \hat{s}_{max}^-$. The proof of $\lim_{s\rightarrow \hat{s}_{max}^-}\hat{t}(s)=1$ is similar to the proof of Theorem \ref{thm2.7}.

$(iv)$ Since $\mathbf{0}$ is a regular value of  $F_{\rm Z}$, all solutions of $F_{\rm Z}(\mathbf{x},\lambda)=\mathbf{0}$ in  $\mathbb{R}^{n+1}_{\geqslant 0}$ are isolated. It follows from Theorem \ref{thm2.8} $(iv)$ that $\lim_{s\rightarrow \hat{s}_{max}^-}(\hat{\mathbf{x}}(s),\hat{\lambda}(s))=(\hat{\mathbf{x}}_*,\hat{\lambda}_*)$. It is easily seen that $(\hat{\mathbf{x}}_*,\hat{\lambda}_*)$ is a Z-eigenpair of $\mathcal{A}$ and $(\hat{\mathbf{x}}_*,\hat{\lambda}_*)\in \mathbb{R}_{> 0}^{n+1}$ because  $\mathcal{A}$ be irreducible. Since  $\mathbf{0}$ is a regular value of  $F_{\rm Z}$, we have $(\hat{\mathbf{x}}_*,\hat{\lambda}_*)\neq (\mathbf{x}_{*,1},\lambda_{*,1})$ and $(\hat{\mathbf{x}}_*,\hat{\lambda}_*)\neq (\mathbf{x}_{*,2},\lambda_{*,2})$. Using the Homotopy Invariance of Degree theorem (Theorem \ref{thm2.9}), we have
${\rm Sgn}({\rm det }(\mathscr{D}_{\mathbf{x},\lambda}F_{\rm Z}(\hat{\mathbf{x}}_*,\hat{\lambda}_*)))=-{\rm Sgn}({\rm det }(\mathscr{D}_{\mathbf{x},\lambda}F_{\rm Z}(\mathbf{x}_{*,2},\lambda_{*,2})))=(-1)^n.$
\end{proof}

Now, we can develop an algorithm for computing an odd number of positive $Z
$-eigenpairs of an irreducible nonnegative tensor $\mathcal{A}$. The
flowchart of this algorithm is shown in Figure \ref{fig1}.

\begin{figure}[htb]
\centering \resizebox{5in}{!}{\includegraphics{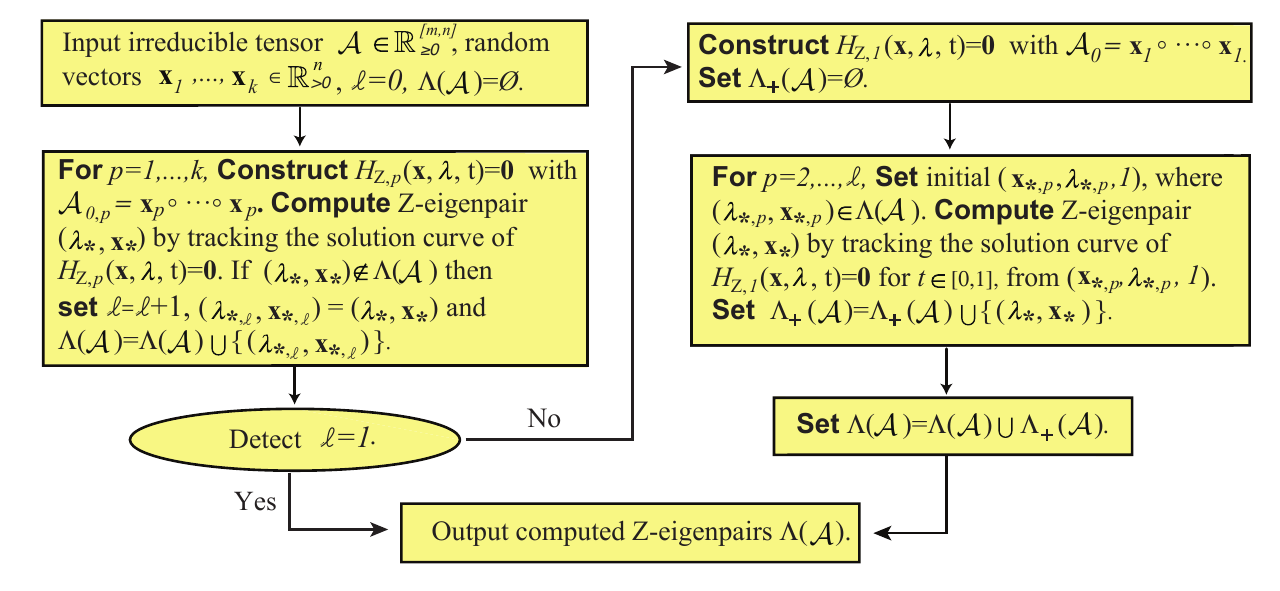}}
\caption{The flowchart of the pseudo-arclength continuation method for
computing an odd number of positive Z-eigenpairs of an irreducible nonnegative tensor $\mathcal{A}$.}
\label{fig1}
\end{figure}

\begin{Remark}
\label{rem4.2} When $\mathcal{A}\in \mathbb{R}_{\geqslant 0}^{[m,n]}$ is
irreducible, if all solutions in $\mathbb{R}_{\geqslant 0}^{n+1}$ of $F_{\rm Z}(%
\mathbf{x}, \lambda)=\mathbf{0}$ are isolated, then the algorithm shown in Figure \ref{fig1} is
guaranteed to compute an odd number of positive Z-eigenpairs, counting
multiplicities. In addition, if $\mathbf{0}\in \mathbb{R}^{n+1}$ is a regular value of $F_{\rm Z}$, then those positive Z-eigenpairs are distinct.
\end{Remark}

\subsection{Parameter continuation method for computing H-eigenpair of nonnegative tensors}

Given a nonnegative tensor $\mathcal{A}\in \mathbb{R}^{[m,n]}_{\geqslant 0}$%
. Let $\mathcal{A}_0\in \mathbb{R}_{>0}^{[m,n]}$ in \eqref{eq2.1} be a
rank-1 tensor and let $(\lambda_0,\mathbf{x}_0)$ in \eqref{eqlamx} be the
positive H-eigenpair of $\mathcal{A}_0$. Theorem \ref{thm2.2} shows that $%
C_{(\mathbf{x}_0,\lambda_0)}$ defined in \eqref{eq2.4} is the solution set
of the homotopy $H_{\rm H}(\mathbf{x},\lambda,t)=\mathbf{0}$ in $\mathbb{R}%
^{n+1}_{\geqslant 0}\times [0,1)$, where the function $H_{\rm H}$ is defined in %
\eqref{eq2.2}. Here, the solution set $C_{(\mathbf{x}_0,\lambda_0)}$ can be
parameterized by $t\in [0,1)$. In addition, Theorem \ref{thm2.3} shows that if $F_{\rm H}(\mathbf{x},\lambda)=\mathbf{0}$ has only isolated solution in $\mathbb{R}_{\geqslant 0}^{n+1}$, then a nonnegative H-eigenpair of $\mathcal{A}$ can be computed by tracking the curve $C_{(\mathbf{x}_0,\lambda_0)}$. It is natural to employ parameter
continuation method for tracking the solution curve of $H_{\rm H}(\mathbf{x}%
,\lambda,t)=\mathbf{0}$ with initial $(\mathbf{x}_0,\lambda_0,0)$. Denote $%
\mathbf{u}=(\mathbf{x}^{\top},\lambda)^{\top}\in \mathbb{R}^{n+1}$. Then $%
H_{\rm H}(\mathbf{u},t)\equiv H_{\rm H}(\mathbf{x},\lambda,t)=\mathbf{0}$. Parameter
continuation method (see \cite{Keller:1987}) takes a prediction-correction
approach. The prediction and correction steps are described as follows.

{\leftmargini=7mm
\leftmarginii=6mm
\begin{itemize}
\item Prediction step: Suppose that $(\mathbf{u}_{i}, t_i)$ is a point lying
(approximately) on a solution curve of $H_{\rm H}(\mathbf{u},t)=\mathbf{0}$. The
Euler predictor
$\mathbf{u}_{i+1,1}=\mathbf{u}_{i}+\Delta t_i \dot{\mathbf{u}}_{i}$
is used to predict a new point.
Here, $\Delta t_i >0$ is a suitable step length satisfying $%
t_{i+1}=t_i+\Delta t_i \leqslant 1$ and $\dot{\mathbf{u}}_{i}$ satisfies the
linear system $\mathscr{D}_{\mathbf{u}}H_{\rm H}(\mathbf{u}_i,t_i)\dot{\mathbf{u}}%
_{i}=-\mathscr{D}_{t}H_{\rm H}(\mathbf{u}_i,t_i)$.

\item Correction step: Let $t=t_{i+1}$ be fixed. We use Newton's method to
compute the approximate solution of $H_{\rm H}(\mathbf{u},t_{i+1})=\mathbf{0}$
with initial value $\mathbf{u}_{i+1,1}$. The iteration $\mathbf{u}%
_{i+1,\ell+1}=\mathbf{u}_{i+1,\ell}+\delta_{\ell}$ is computed for $%
\ell=1,2,\ldots$, where $\delta_{\ell}$ satisfies the linear system $%
\mathscr{D}_{\mathbf{u}}H_{\rm H}(\mathbf{u}_{i+1,\ell},t_{i+1})\delta_{\ell}=-H_{\rm H}(%
\mathbf{u}_{i+1,\ell},t_{i+1})$. If $\{\mathbf{u}_{i+1,\ell}\}$ converges
until $\ell=\ell_{\infty}$, then we set $\mathbf{u}_{i+1}=\mathbf{u}%
_{i+1,\ell_{\infty}}$ and accept $(\mathbf{u}_{i+1},t_{i+1})$ as a new
approximation to the solution curve of $H_{\rm H}(\mathbf{u},t)=\mathbf{0}$.
\end{itemize}}

\begin{Remark}
\label{rem4.3} If $\mathcal{A}\geqslant 0$ is weakly irreducible, then
Theorem \ref{thm2.4} shows that we can compute the unique positive H-eigenpair, $(\lambda_*,\mathbf{x}_*)$, of $\mathcal{A}$ by tracking the
solution curve. Note that the positive H-eigenvalue $\lambda_*$ is the
largest H-eigenvalue of $\mathcal{A}$.
\end{Remark}

\subsection{ Comparison to other methods}

In this section, we compare the numerical schemes, SS-HOPM (for Z%
-eigenpair) and NQZ, NNI (for H-eigenpair), with continuation method. The
main difference is that those three schemes are iteration methods. For the
computational complexity, SS-HOPM, NQZ and NNI require one evaluation of a
vector in the form $\mathcal{A}\mathbf{x}^{m-1}$ for each iteration. Thus,
the computational complexity for each iteration is $O(n^m)$.

Continuation methods are guaranteed to compute nonnegative Z-eigenpair and
H-eigenpairs of a tensor $\mathcal{A}\in \mathbb{R}_{\geqslant 0}^{[m,n]}$%
. In prediction and correction steps of continuation method, we need
evaluate the Jacobian matrix and residual of homotopy function, which is
the dominant computational complexity in continuation method. With the help
of rank-1 tensor $\mathcal{A}_0$ in \eqref{eq2.1}, we can compute $\mathcal{A%
}(t_*)\mathbf{x}_*^{m-1}$ as follows:
\begin{align}
\mathcal{A}(t_*)\mathbf{x}_*^{m-1}=(1-t_*)\left(\prod_{k=2}^m(\mathbf{x}_k^{\top}\mathbf{x}%
_*)\right)\mathbf{x}_1+t_*\mathcal{A}\mathbf{x}_*^{m-1}.  \label{eq4.1}
\end{align}
The computational complexities of $\mathcal{A}(t_*)\mathbf{x}_*^{m-1}$ and
of $\mathcal{A}\mathbf{x}_*^{m-1}$ are almost the same, which are $O(n^m)$.
The Jacobian matrix $A_{t_*}\equiv \mathscr{D}_{\mathbf{x}}(\mathcal{A}(t_*)\mathbf{x}%
^{m-1})|_{\mathbf{x} =\mathbf{x}_*}$ requires to compute matrices $\mathcal{A}%
\times_2\mathbf{x}_*\cdots \times_{k-1}\mathbf{x}_*\times_{k+1}\mathbf{x}%
_*\cdots \times_m\mathbf{x}_*$ for $k=2,\ldots,m$. If the tensor $\mathcal{A}
$ is semi-symmetric,\footnote{$\mathcal{A}\in \mathbb{R}^{[m,n]}$ is called
semi-symmetric if $\mathcal{A}_{i,j_2,\cdots, j_m}=\mathcal{A}%
_{i,i_2,\cdots, i_m}$, where $1\leqslant i_1\leqslant n$, $j_2,\cdots, j_m$
is any permutation of $i_2,\cdots, i_m$, $1\leqslant i_2,\cdots,
i_m\leqslant n$.} then we have $\mathscr{D}_{\mathbf{x}}(\mathcal{A}\mathbf{x%
}^{m-1})|_{\mathbf{x}=\mathbf{x}_*}=(m-1)\mathcal{A}\mathbf{x}_*^{m-2}$, which is a
precursor of $\mathcal{A}\mathbf{x}_*^{m-1}$. Note that a symmetric tensor%
\footnote{$\mathcal{A}\in \mathbb{R}^{[m,n]}$ is called symmetric if $%
\mathcal{A}_{j_1,j_2,\cdots, j_m}=\mathcal{A}_{i_1,i_2,\cdots, i_m}$, where $%
j_1,j_2,\cdots, j_m$ is any permutation of $i_1, i_2,\cdots, i_m$, for $%
1\leqslant i_1, i_2,\cdots, i_m\leqslant n$.} is semi-symmetric. When $%
\mathcal{A}$ is semi-symmetric, we may choose the rank-1 tensor $\mathcal{A}%
_0=\mathbf{x}_1\circ\cdots\circ\mathbf{x}_1\in \mathbb{R}^{[m,n]}_{>0}$ as a
symmetric tensor. Then the computational complexity of each prediction step
or each iteration of Newton's method in correction step of continuation
method is $O(n^m)$ by using the formula \eqref{eq4.1}. When $\mathcal{A}\in
\mathbb{R}_{\geqslant 0}^{[m,n]}$ is not a semi-symmetric, Ni and Qi \cite%
{Ni-Qi:2015} shown that there exists a semi-symmetric $\mathcal{A}_s\in
\mathbb{R}_{\geqslant 0}^{[m,n]}$ such that $\mathcal{A}\mathbf{x}^{m-1}=%
\mathcal{A}_s\mathbf{x}^{m-1}$ for each $\mathbf{x}\in \mathbb{R}^{n}$. The
computational complexity of constructing the semi-symmetric $\mathcal{A}%
_s\in \mathbb{R}_{\geqslant 0}^{[m,n]}$ is $O(n^m)$. Hence, it is more
efficient if we replace the tensor $\mathcal{A} $ to a semi-symmetric $%
\mathcal{A}_s\in \mathbb{R}_{\geqslant 0}^{[m,n]}$ before employing
continuation method.

In the following, we itemized the sufficient conditions for the convergence
of numerical schemes, SS-HOPM, NQZ, NNI and continuation method.

{\leftmargini=7mm
\leftmarginii=6mm
\begin{itemize}
\item For computing Z-eigenpairs of a tensor $\mathcal{A}\in \mathbb{R}^{[m,n]}$:

\begin{itemize}
\item SS-HOPM \cite{Kolda-Mayo:2011} is guaranteed to compute the Z-eigenpairs of a \textit{real symmetric tensor} $\mathcal{A}$, which is closely related to optimal rank-1 approximation of $\mathcal{A}$. In addition, if $\mathcal{A} \in \mathbb{R}_{\geqslant 0}^{[m,n]}$
is nonnegative symmetric, then SS-HOPM is guaranteed to find a
nonnegative Z-eigenpair of $\mathcal{A}$. The convergence of SS-HOPM
appears to be linear.

\item Continuation method is guaranteed to find a nonnegative Z-eigenpair of $\mathcal{A} \in \mathbb{R}_{\geqslant 0}^{[m,n]}$ if $F_{\rm Z}(\mathbf{x},\lambda)=\mathbf{0}$ has only
isolated solution in $\mathbb{R}^{n+1}_{\geqslant 0}$ (see Theorem \ref%
{thm2.8}).
\end{itemize}


\item For computing H-eigenpair of a tensor $\mathcal{A}\in \mathbb{R}%
_{\geqslant 0}^{[m,n]}$:

\begin{itemize}
\item NQZ \cite{Ng-Qi-Zhou:2009,Zhang-Qi:2012} is a power method for
computing the largest H-eigenvalue of $\mathcal{A}$. The convergence of
NQZ appears to be linear for weakly primitive tensors.

\item NNI \cite{Liu-Guo-Lin:2016,Liu-Guo-Lin:2017} is guaranteed to compute
the largest H-eigenvalue of a weakly irreducible nonnegative tensor. The
convergence rate is quadratic when it is near convergence. However, the
initial monotone convergence of NNI may be quite slow.

\item Continuation method is guaranteed to compute the largest H%
-eigenvalue of $\mathcal{A}$ if all solutions of $F_{\rm H}(\mathbf{x},\lambda)=%
\mathbf{0}$ in $\mathbb{R}^{n+1}_{\geqslant 0}$ are isolated.
\end{itemize}

\end{itemize}}

Note that if $\mathcal{A}$ is weakly primitive then $\mathcal{A}$ is weakly
irreducible and if $\mathcal{A}$ is weakly irreducible then the solution of $%
F_{\rm H}(\mathbf{x},\lambda)=\mathbf{0}$ in $\mathbb{R}^{n+1}_{\geqslant 0}$ is
unique and isolated.

\section{Numerical experiments}

In this section, we present some numerical results to support our theory.
All numerical tests were performed using MATLAB 2014a on a Mac Pro with 3.7
GHz Quad-Core Intel Xeon E5 and 32 GB memory. In the following numerical
results, ``Steps" denotes the number of steps (a step $=$ a prediction step $%
+$ a correction step) of continuation method to achieve the solution,
``\#(Eval)" denotes the number of evaluations of $\mathcal{A}\mathbf{x}^{m-1}
$, ``Res" denotes the residual, $\|F_{\rm Z}(\mathbf{x}_*, \lambda_*)\|$ (or $%
\|F_{\rm H}(\mathbf{x}_*, \lambda_*)\|$), when the Z-eigenpair (or H-eigenpair), $(\lambda_*,\mathbf{x}_*)$, is computed and \#(TP) denotes the number of turning points of the solution curve.
The maximum number of evaluations allowed is 2000 for NQZ, SS-HOPM and NNI.

\subsection{Numerical results for computing Z-eigenpairs}

We first apply continuation method to compute Z-eigenpairs of the $m$%
th-order $n$-dimensional signless Laplacian tensor \cite{HQ15,HQX15}.

\begin{Example}
\label{ex5.1} Consider the signless Laplacian tensor $\mathcal{A} = \mathcal{%
D} + \mathcal{C}\in \mathbb{R}_{\geqslant 0}^{[m,n]}$ of an $m$-uniform
connected hypergraph \cite{HQ15,HQX15}, where $\mathcal{D}$ is the diagonal
tensor with diagonal element $\mathcal{D}_{i,\cdots,i}$ equal to the degree
of vertex $i$ for each $i$, and $\mathcal{C}$ is the adjacency tensor
defined in \cite{HQ15,HQX15,Cooper-Dutle:2012} which is symmetric. Consider the edge set $E=\left \{ \{i-m+2,
i-m+3,\ldots,i, i+1\}, i=m-1, \ldots,n\right\}$ in \cite{Liu-Guo-Lin:2017}, where $n+1$ is identified
with $1$. The corresponding tensor $\mathcal{A} $ is weakly primitive (and
thus weakly irreducible).


Given a signless Laplacian tensor $\mathcal{A}\in \mathbb{R}_{\geqslant
0}^{[m,n]}$, let $\mathcal{A}_0=\mathbf{x}%
_1\circ\cdots\circ\mathbf{x}_1$, where $\mathbf{x}_1\in \mathbb{R}^n_{>0}$
is generic with $\|\mathbf{x}_1\|\in[0.9,1.1]$. It follows from \eqref{eqlamxz} that  $(\lambda_0,
\mathbf{x}_0)=(\|\mathbf{x}_1\|^m,\frac{\mathbf{x}_1}{\|\mathbf{x}_1\|})$ is the unique positive Z-eigenpair of $\mathcal{A}_0$.
Table \ref{tbl:outer2} reports the results obtained by tracking the solution
curve of $H_{\rm Z}(\mathbf{x},\lambda,t)=\mathbf{0}$ by pseudo-arclength
continuation method for various of $m$ and $n$. From Table \ref{tbl:outer2},
we can see that the solution curve of $H_{\rm Z}(\mathbf{x},\lambda,t)=\mathbf{0}$
with initial $(\mathbf{x}_0,\lambda_0,0)$ has two turning points for each
test case. The numbers of Steps and \#(Eval) increase when the distance
between those two turning points increases. In this example, the number of
evaluations, \#(Eval), is at most 176. Figure \ref%
{fig2} shows the bifurcation diagram of the solution curve of $H_{\rm Z}(\mathbf{x}%
,\lambda,t)=\mathbf{0}$ for the case $m=5$ and $n=20$. The corresponding
eigenvectors, $\mathbf{x}(s)$, are attached near to the solution curve.
\begin{table}[htbp]
\scriptsize
\caption{Numerical results for Example \protect\ref{ex5.1}.}
\label{tbl:outer2}\centering
\begin{tabular}{rrlrcccrr}
\hline
\multicolumn{2}{c}{Tensor $\mathcal{A}$} & \multicolumn{7}{c}{Continuation
method} \\ \cline{1-2}\cline{4-9}
$m$ & $n $ &  & Steps & \#(Eval) & Res & \#(TP) & \multicolumn{2}{c}{turning
points ($t$)} \\ \hline
$3$ & $20$ &  & $12$ & $67$ & 7.70e-11 & 2 & $0.405$ & 0.385 \\
$3$ & $50$ &  & $18$ & $94$ & 2.68e-12 & 2 & $0.391$ & 0.309 \\
$3$ & $100$ &  & $20$ & $106$ & 2.26e-20 & 2 & $0.444$ & 0.210 \\ \hline
$4$ & $20$ &  & $18$ & $91$ & 1.64e-16 & 2 & $0.276$ & 0.177 \\
$4$ & $50$ &  & $28$ & $128$ & 3.42e-11 & 2 & $0.372$ & 0.0577 \\
$4$ & $100$ &  & $34$ & $160$ & 4.52e-11 & 2 & $0.642$ & 0.0666 \\ \hline
$5$ & $20$ &  & $22$ & $103$ & 4.90e-18 & 2 & $0.244$ & 0.0781 \\
$5$ & $50$ &  & $32$ & $148$ & 1.48e-12 & 2 & $0.662$ & 0.0276 \\
$5$ & $100$ &  & $40$ & $176$ & 1.47e-13 & 2 & $0.846$ & 0.0105 \\ \hline
\end{tabular}%
\end{table}
\begin{figure}[htb]
\centering \resizebox{4in}{!}{\includegraphics{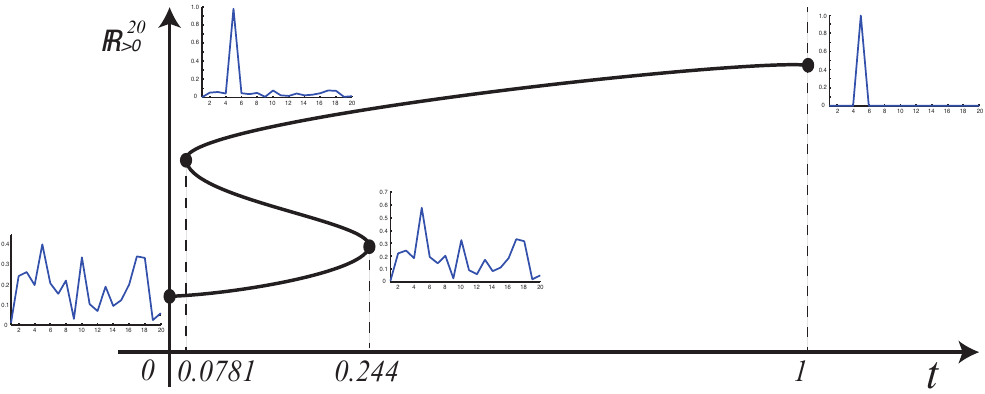}}
\caption{The bifurcation diagram of the solution curve of $H_{\rm Z}(\mathbf{x},%
\protect\lambda,t)=\mathbf{0}$ with $\mathcal{A}\in \mathbb{R}_{\geqslant
0}^{[5,20]}$. The corresponding eigenvectors are attached near to the
solution curve (Example \protect\ref{ex5.1}).}
\label{fig2}
\end{figure}

Corollary \ref{cor2.11} shows that the number of positive Z-eigenpairs of
an irreducible tensor $\widehat{\mathcal{A}}$, counting multiplicities, is
odd. Since the tensor $\mathcal{A}$ constructed in this example is weakly
irreducible, we set $\widehat{\mathcal{A}}=\mathcal{A}+10^{-5}\mathcal{E}$,
where $\mathcal{E}$ is the tensor with all entries equal to 1. Employ the
algorithm shown in Figure \ref{fig1} to the irreducible tensor $\widehat{%
\mathcal{A}}$. In the following numerical tests, we consider the case $m=4$
and $n=20$. For a fixed tensor $\widehat{\mathcal{A}}\in \mathbb{R}%
_{>0}^{[4,20]}$, we run 100 trials of the algorithm using $k$ random initial
vectors $\mathbf{x}_1,\ldots, \mathbf{x}_k\in \mathbb{R}^{20}_{>0}$.
Figure \ref{fig4}  reports the number of occurrences (over 100 trials)  for the  numbers of computed positive Z-eigenpairs of $\widehat{\mathcal{A}}$ in terms of $k=50$ and $70$.

\begin{figure}[htb]
\centering \resizebox{4in}{!}{\includegraphics{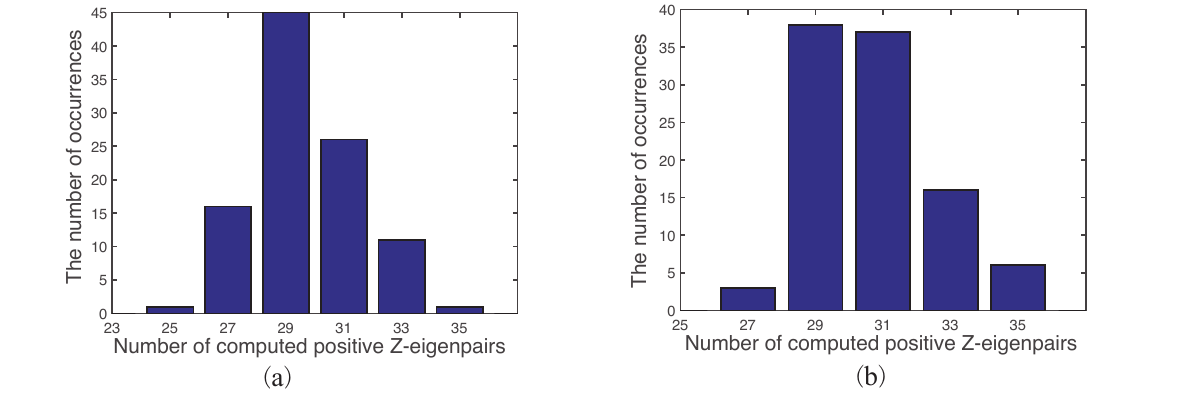}}
\caption{The number of occurrences (over 100 trials)  for the  numbers of computed positive Z-eigenpairs of $\widehat{\mathcal{A}}\in \mathbb{R}%
_{>0}^{[4,20]}$ by using $k=50$ (a) and $k=70$ (b) random vectors (Example \protect\ref{ex5.1}).}
\label{fig4}
\end{figure}

\end{Example}



\begin{Example}
\label{ex5.2} Consider the symmetric tensor $\mathcal{A}(w) = \mathcal{D} + w\mathcal{C}%
\in \mathbb{R}_{\geqslant 0}^{[4,20]}$, where $w\in \mathbb{R}_{>0}$, $%
\mathcal{D}$ and $\mathcal{C}$ are defined in Example \ref{ex5.1}.
We employ continuation method and SS-HOPM with shift parameter $\alpha\in
\mathbb{R}$ to compute positive Z-eigenpair of $\mathcal{A}$. Suppose that
$(\lambda_*,\mathbf{x}_*)\in \mathbb{R}^{21}_{\geqslant 0}$ is a Z%
-eigenpair of $\mathcal{A}(w)$, then $\lambda_*=\lambda_*\mathbf{x}_*^{\top}%
\mathbf{x}_*=\mathbf{x}_*^{\top}\mathcal{A}(w)\mathbf{x}_*^3=:\mathcal{A}(w)%
\mathbf{x}_*^4$. The algorithm SS-HOPM \cite{Kolda-Mayo:2011} is guaranteed
to converge to a local maximum of the optimization problem:
\begin{align}  \label{eq5.1}
\max_{\mathbf{x}\in \mathbb{R}^{20},\ \|\mathbf{x}\|=1}\mathcal{A}(w)\mathbf{%
x}^4
\end{align}
if the shift $\alpha>\beta(\mathcal{A}(w) )$, where the constant $\beta(%
\mathcal{A}(w) )$ is dependent on tensor $\mathcal{A}(w)$. Lemma 4.1 in \cite%
{Kolda-Mayo:2011} shows that $\gamma(w)=72(1+w)$ is an upper bound of $\beta(%
\mathcal{A}(w) )$. Choosing $\alpha >\gamma(w)$ is guaranteed to work but
may slow down convergence. In our numerical experiments, we choose $w=1,3$
and $5$. Note that when $w=1,3$ and $5$ then $\gamma(w)=144, 288$ and $432$,
respectively. Table \ref{tbl:outer5} reports the results obtained by
continuation method and SS-HOPM with $\alpha=\gamma(w)+1$ and $\alpha=1$ in
terms of $w=1,3$ and $5$, where we terminate the iteration of SS-HOPM when $%
\mathrm{Res}<10^{-10}$. In this table, we can see that a local
maximum value, $\lambda_*=\mathcal{A}(w)\mathbf{x}_*^4$, of the optimization
problem \eqref{eq5.1} can also be computed by continuation method. The
number of evaluations, \#(Eval), of continuation method is at most 114 that
is much less than the number of evaluations of SS-HOPM with $%
\alpha=\gamma(w)+1$. SS-HOPM with $\alpha=1$ works for this example, but
there is no theory to guarantee the convergence.

\begin{table}[htbp]
\scriptsize
\caption{Numerical results for Example \protect\ref{ex5.2}.}
\label{tbl:outer5}\centering
\begin{tabular}{r|rrrrrrrr||rrr}
\hline
& \multicolumn{3}{c}{Continuation method} &  & \multicolumn{3}{c}{SS-HOPM$%
(\alpha=\gamma(w)+1)$} &  & \multicolumn{3}{c}{SS-HOPM$(\alpha=1)$}   \\
\cline{2-4}\cline{6-8}\cline{10-12}
$w$ & \#(Eval) & $\lambda_*$ & Res &  & \#(Eval) & $\lambda_*$ & Res &  &
\#(Eval) & $\lambda_*$ & Res   \\ \hline
1 & $91$ & 4 & 1.64e-16 &  & 1211 & 4 & 9.78e-11 &  & 24 & 4 & 8.03e-11   \\
3 & $114$ & 4 & 5.56e-19 &  & 2000 & 4 & 6.61e-07 &  & 33 & 4 & 9.36e-11  \\
5 & 112 & 2.95 & 5.73e-17 &  & 2000 & 2.95 & 2.98e-04 &  & 2000 & 2.95 & 4.27e-05    \\ \hline
\end{tabular}%
\end{table}
\end{Example}

The next example, we consider a \textit{multilinear PageRank} problem
provided in \cite{Gleich-Lim-Yu:2015}. In multilinear PageRank problem, it
needs to compute the positive Z-eigenpair of a stochastic transition
tensor,
\begin{align}  \label{eq5.2}
\mathcal{A}(\alpha)=\alpha\mathcal{P}+(1-\alpha)\mathbf{v}\circ\mathbf{e}%
\circ\cdots \circ\mathbf{e}\in\mathbb{R}_{\geqslant 0}^{[m,n]},
\end{align}
where $\mathcal{P}$ is the transition tensor of the higher-order Markov
chain, $\mathbf{v}\in \mathbb{R}_{\geqslant 0}^n$ is a stochastic vector, $%
\mathbf{e}=[1,1,\cdots,1]^{\top}\in \mathbb{R}^n$ and $\alpha\in (0,1)$.

\begin{Example}
\label{ex5.2.5} We consider stochastic transition $\mathcal{A}(\alpha)\in%
\mathbb{R}_{\geqslant 0}^{[3,6]}$ has the form in \eqref{eq5.2}, where the
unfolding of tensor $\mathcal{P}$ is {\scriptsize
\begin{align*}
&\left[%
\begin{array}{c|c|c|c|c|c}
{\mathcal{P}(:,:,1)} & {\mathcal{P}(:,:,2)} & {\mathcal{P}(:,:,3)} & {%
\mathcal{P}(:,:,4)} & {\mathcal{P}(:,:,5)} & {\mathcal{P}(:,:,6)}%
\end{array}%
\right] \\
&= \left[%
\begin{array}{@{\hspace{0.1cm}}c@{\hspace{0.1cm}}c@{\hspace{0.1cm}}c@{\hspace{0.1cm}}c@{\hspace{0.1cm}}c@{\hspace{0.1cm}}c|c@{\hspace{0.1cm}}c@{\hspace{0.1cm}}c@{\hspace{0.1cm}}c@{\hspace{0.1cm}}c@{\hspace{0.1cm}}c|c@{\hspace{0.1cm}}c@{\hspace{0.1cm}}c@{\hspace{0.1cm}}c@{\hspace{0.1cm}}c@{\hspace{0.1cm}}c|c@{\hspace{0.1cm}}c@{\hspace{0.1cm}}c@{\hspace{0.1cm}}c@{\hspace{0.1cm}}c@{\hspace{0.1cm}}c|c@{\hspace{0.1cm}}c@{\hspace{0.1cm}}c@{\hspace{0.1cm}}c@{\hspace{0.1cm}}c@{\hspace{0.1cm}}c|c@{\hspace{0.1cm}}c@{\hspace{0.1cm}}c@{\hspace{0.1cm}}c@{\hspace{0.1cm}}c@{\hspace{0.1cm}}c}
0 & 0 & 0 & 0 & 0 & 0 & 0 & 0 & 0 & 0 & 0 & 1 & 1 & 0 & 0 & 0 & 0 & 0 & 0 & 0
& 1 & 0 & 1 & 0 & 0 & 1 & 1 & 0 & 0 & 0 & 0 & 0 & 0 & 0 & 0 & 1 \\
0 & 0 & 0 & 0 & 0 & 0 & 0 & 0 & 1 & 0 & 0 & 0 & 0 & 1 & 0 & 0 & 0 & 0 & 0 & 0
& 0 & 0 & 0 & 0 & 0 & 1 & 0 & 0 & 1 & 0 & 0 & 0 & 0 & 0 & 0 & 0 \\
0 & 0 & 0 & 0 & 0 & 0 & 0 & 0 & 0 & 0 & 0 & 0 & 1 & 0 & 0 & 0 & 0 & 0 & 0 & 0
& 0 & 0 & 0 & 0 & 1 & 0 & 0 & 0 & 1 & 0 & 1 & 1 & 0 & 0 & 1 & 0 \\
0 & 0 & 0 & 0 & 0 & 0 & 0 & 0 & 0 & 0 & 0 & 0 & 1 & 0 & 0 & 0 & 0 & 0 & 0 & 0
& 0 & 1 & 0 & 1 & 0 & 0 & 0 & 0 & 0 & 0 & 0 & 1 & 0 & 1 & 0 & 0 \\
0 & 0 & 0 & 0 & 0 & 0 & 0 & 1 & 0 & 0 & 0 & 1 & 0 & 0 & 0 & 0 & 0 & 0 & 0 & 1
& 0 & 0 & 0 & 0 & 0 & 0 & 0 & 1 & 0 & 0 & 0 & 0 & 1 & 0 & 1 & 0 \\
1 & 1 & 1 & 1 & 1 & 1 & 1 & 0 & 0 & 1 & 1 & 0 & 1 & 0 & 1 & 1 & 1 & 1 & 1 & 0
& 0 & 0 & 0 & 0 & 0 & 0 & 0 & 0 & 0 & 1 & 0 & 0 & 0 & 0 & 0 & 0%
\end{array}%
\right],
\end{align*}}
and the stochastic vector $\mathbf{v}=\mathbf{e}/6$. SS-HOPM and  Newton method fail to converge the nonnegative Z-eigenpair when $\alpha=0.99$ (see \cite{Gleich-Lim-Yu:2015}).
We employ
continuation method to compute the positive Z-eigenpair of $\mathcal{A}(\alpha)$ in
terms of $\alpha=0.9$, $0.99$ and $0.999$.  Table \ref{tbl:outer6} reports
the numerical results. This table shows that when $\alpha=0.99$ and $0.999$,
the solution curves have two turning points, but no turning point occur when
$\alpha=0.9$.
\begin{table}[htbp]
\scriptsize
\caption{Continuation method for positive Z-eigenpair of $\mathcal{A}(%
\protect\alpha)$ (Example \protect\ref{ex5.2.5}).}
\label{tbl:outer6}\centering
\begin{tabular}{r|rccccc}
\hline
$\alpha $ & Steps & \#(Eval) & Res & \#(TP) & \multicolumn{2}{c}{turning
points ($t$)} \\ \hline\hline
$0.9$ & $4$ & $29$ & 3.55e-16 & 0 & - & - \\
$0.99$ & $13$ & $74$ & 1.13e-16 & 2 & $0.999$ & 0.952 \\
$0.999$ & $16$ & $90$ & 1.11e-16 & 2 & $0.984$ & 0.849 \\ \hline
\end{tabular}%
\end{table}
\end{Example}

In the following example, we consider a small size irreducible tensor $%
\mathcal{A}\in \mathbb{R}_{\geqslant 0}^{[4,2]}$, which has three positive $Z
$-eigenpairs. This tensor is provided in \cite{Chang-Pearson-Zhang:2013}.

\begin{Example}
\label{ex5.3} Let $\mathcal{A}\in \mathbb{R}_{\geqslant 0}^{[4,2]}$ be
defined by
\begin{align*}
\mathcal{A}_{1111}&=\mathcal{A}_{2222}=\frac{4}{\sqrt{3}},\ \ \mathcal{A}%
_{1112}=\mathcal{A}_{1211}=\mathcal{A}_{2111}=1, \\
\mathcal{A}_{1222}&=\mathcal{A}_{2122}=\mathcal{A}_{2212}=\mathcal{A}%
_{2221}=1, \text{ and }\mathcal{A}_{ijkl}=0 \text{ elsewhere}.
\end{align*}
Obviously, $\mathcal{A}$ is irreducible.
The system of polynomials $F_{Z}$ in \eqref{eqZ} has the form
\begin{align*}
F_{\rm Z}(\mathbf{x},\lambda)=\left(%
\begin{array}{l}
\frac{4}{\sqrt{3}}x_1^3+3x_1^2x_2+x_2^3-\lambda x_1 \\
\frac{4}{\sqrt{3}}x_2^3+3x_1x_2^2+x_1^3-\lambda x_2 \\
x_1^2+x_2^2-1%
\end{array}%
\right)=\mathbf{0},
\end{align*}
where $\mathbf{x}=(x_1,x_2)^{\top}$.  \cite{Chang-Pearson-Zhang:2013} shown
that $\mathcal{A}$ has there positive Z-eigenpairs:
{\leftmargini=8mm
\begin{itemize}
\item $\hat{\lambda}_0=2+\frac{2}{\sqrt{3}}\approx 3.1547$ with corresponding
positive Z-eigenvector $\hat{\mathbf{x}}_0=\left[\frac{\sqrt{2}}{2},\frac{\sqrt{2%
}}{2}\right]^{\top}$;

\item $\hat{\lambda}_1=\hat{\lambda}_2=\frac{11}{2\sqrt{3}}\approx 3.1754$ with
corresponding positive Z-eigenvectors $\hat{\mathbf{x}}_1=\left[\frac{\sqrt{3}}{2%
},\frac{1}{2}\right]^{\top}$ and $\hat{\mathbf{x}}_2=\left[\frac{1}{2},\frac{\sqrt{%
3}}{2}\right]^{\top}$.
\end{itemize}}
That is, $F_{Z}(\hat{\mathbf{x}}_0,\hat{\lambda}_0)=F_{Z}(\hat{\mathbf{x}}_1,\hat{\lambda}_1)=F_{Z}(%
\hat{\mathbf{x}}_2,\hat{\lambda}_2)=\mathbf{0}$. The Jacobian matrix of $F_{\rm Z}$ is
\begin{align*}
\mathscr{D}_{\mathbf{x},\lambda}F_{\rm Z}(\mathbf{x},\lambda)=\left[%
\begin{array}{ccc}
4\sqrt{3}x_1^2+6x_1x_2-\lambda & 3x_1^2+3x_2^2 & -x_1 \\
3x_1^2+3x_2^2 & 4\sqrt{3}x_2^2+6x_1x_2-\lambda & -x_2 \\
2x_1 & 2x_2 & 0%
\end{array}%
\right].
\end{align*}
Then we have
\begin{align*}
&\mathrm{Sgn}(\mathrm{det }(\mathscr{D}_{\mathbf{x},\lambda}F_{\rm Z}(\hat{\mathbf{x}}%
_0,\hat{\lambda}_0)))=1, \text{ and } \\
&\mathrm{Sgn}(\mathrm{det }(\mathscr{D}_{\mathbf{x},\lambda}F_{\rm Z}(\hat{\mathbf{x}}%
_1,\hat{\lambda}_1)))=\mathrm{Sgn}(\mathrm{det }(\mathscr{D}_{\mathbf{x}%
,\lambda}F_{\rm Z}(\hat{\mathbf{x}}_2,\hat{\lambda}_2)))=-1,
\end{align*}
and hence, $\mathrm{deg}(F_{\rm Z}, \mathbb{R}^{3}_{> 0},\mathbf{0})\equiv
\sum_{k=0}^2\mathrm{Sgn}(\mathrm{det }(\mathscr{D}_{\mathbf{x},\lambda}F_{\rm Z}(%
\hat{\mathbf{x}}_k,\hat{\lambda}_k)))=-1$. This result has been shown in Theorem \ref%
{thm2.10} with $n=2$. For any rank-1 symmetric tensor $\mathcal{A}_0\in
\mathbb{R}_{>0}^{[4,2]}$, $\mathrm{deg}(F^0_{\rm Z}, \mathbb{R}^{3}_{> 0},\mathbf{0}%
)=-1$ (see Lemma \ref{lem3.9}), where $F^0_{\rm Z}$ is defined in \eqref{eqZ}. From
Theorem \ref{thm2.11} $(i)$, we can only compute Z-eigenpairs, $(\hat{\lambda}_1,%
\hat{\mathbf{x}}_1)$ or $(\hat{\lambda}_2,\hat{\mathbf{x}}_2)$, by tracking the solution curve
of $H_{\rm Z}(\mathbf{x},\lambda,t )=\mathbf{0}$. Let $\mathcal{A}%
_{0,1},\ \mathcal{A}_{0,2}\in \mathbb{R}_{>0}^{[4,2]}$ be rank-1 symmetric
tensors and two homotopy equations $H_{\rm Z,1}(\mathbf{x},\lambda,t )=\mathbf{0}$
and $H_{\rm Z,2}(\mathbf{x},\lambda,t )=\mathbf{0}$ be constructed in \eqref{eq2.7}. Suppose that the Z-eigenpairs, $(\hat{\lambda}_1,\hat{\mathbf{x}}_1)$ and $%
(\hat{\lambda}_2,\hat{\mathbf{x}}_2)$, can be computed by tracking the solution curves
of $H_{\rm Z,1}(\mathbf{x},\lambda,t )=\mathbf{0}$ and $H_{\rm Z,2}(\mathbf{x},\lambda,t
)=\mathbf{0}$, respectively. Theorem \ref{thm2.11} $(iv)$ shows that a new
positive Z-eigenpair $(\lambda_*,\mathbf{x}_*)\in \mathbb{R}_{> 0}^3$ can
be computed by tracking the solution curve of $H_{\rm Z,1}(\mathbf{x},\lambda,t )=\mathbf{0}$ with initial $(\hat{\mathbf{x}}_2,\hat{\lambda}_2,1)$ and $\mathrm{Sgn}(%
\mathrm{det }(\mathscr{D}_{\mathbf{x},\lambda}F_{\rm Z}(\mathbf{x}_*,\lambda_*)))=1
$. Hence, $(\lambda_*,\mathbf{x}_*)=(\hat{\lambda}_0,\hat{\mathbf{x}}_0)$. We run 100
trials of the algorithm using $k$ random initial vectors $\mathbf{x}%
_1,\ldots, \mathbf{x}_k\in \mathbb{R}^{2}_{>0}$. Table \ref{tbl:outer4}
reports the number of occurrences (over 100 trials) for the  numbers of computed Z-eigenpairs of $\mathcal{A}$  in terms of $k=2$, $5$ and $8$.
\begin{table}[htbp]
\scriptsize
\caption{The number of occurrences (over 100 trials) for the numbers of computed positive Z-eigenpairs of $\mathcal{A}\in
\mathbb{R}_{>0}^{[4,2]}$ by using $k=2$, $5$ and $8$ random vectors
(Example \protect\ref{ex5.3}).}
\label{tbl:outer4}\centering
\begin{tabular}{c|cc}
\hline
& \multicolumn{2}{c}{No. of computed Z-eigenpairs} \\ \cline{2-3}
$k$ & $1 $ & $3 $ \\ \hline
2 & $49$ & $51$ \\
5 & $6$ & $94$ \\
8 & $0$ & $100$ \\ \hline
\end{tabular}%
\end{table}
\end{Example}
\subsection{Numerical results for computing H-eigenpair}

In this section, we then apply continuation method, NQZ and NNI to compute
the positive H-eigenpair of the $m$th-order $n$-dimensional signless
Laplacian tensor \cite{HQ15,HQX15}.

\begin{Example}
\label{ex5.4} Consider a tensor $\mathcal{A}= \mathcal{D} + \mathcal{C}\in
\mathbb{R}_{\geqslant 0}^{[m,n]}$, where $\mathcal{D}$ and $\mathcal{C}$ are
defined in Example \ref{ex5.1}.
Let  $\mathcal{A}_0=\mathbf{x}%
_1\circ\cdots\circ\mathbf{x}_1$, where $\mathbf{x}_1=\frac{1}{n^{(m-1)/m}}%
[1,\ldots,1]^{\top}\in \mathbb{R}_{>0}^n$. From \eqref{eqlamx}, we obtain
the unique nonzero H-eigenvalue of $\mathcal{A}_0$ is $\lambda_0=%
\prod_{k=2}^{m}(\mathbf{x}_1^{\top}\mathbf{x}_1^{[1/(m-1)]})=\prod_{k=2}^{m}(%
\frac{1}{n^{(m-1)/m}}\cdot \frac{1}{n^{1/m}}\cdot n)=1$ and the associated
unit positive H-eigenvector is $\mathbf{x}_0=\frac{1}{\sqrt{n}}%
[1,\ldots,1]^{\top}$. Table \ref{tbl:outer1} reports the results obtained by
continuation method, NQZ and NNI for
various of $m$ and $n$, where we terminate the iteration of NQZ and NNI when
$\mathrm{Res}<10^{-10}$.

\begin{table}[htbp]
\scriptsize
\caption{Numerical results for Example \protect\ref{ex5.4}.}
\label{tbl:outer1}\centering
\begin{tabular}{rrlrcccrccrc}
\hline
\multicolumn{2}{c}{Tensor $\mathcal{A}$} & \multicolumn{4}{c}{Continuation
method} & \multicolumn{3}{c}{NQZ} & \multicolumn{3}{c}{NNI} \\
\cline{1-2}\cline{4-6}\cline{8-9}\cline{11-12}
$m$ & $n $ &  & Steps & \#(Eval) & Res &  & \#(Eval) & Res &  & \#(Eval) &
Res \\ \hline
$3$ & $20$ &  & $4$ & $18$ & 1.90e-11 &  & $240$ & 9.33e-11 &  & $7$ &
4.50e-16 \\
$3$ & $50$ &  & $7$ & $36$ & 2.88e-12 &  & $1285$ & 9.88e-11 &  & $11$ &
7.43e-16 \\
$3$ & $100$ &  & $11$ & $56$ & 3.42e-12 &  & $2000$ & 3.03e-06 &  & $126$ &
1.24e-12 \\ \hline
$4$ & $20$ &  & $4$ & $18$ & 2.94e-14 &  & $138$ & 9.03e-11 &  & $7$ &
7.28e-11 \\
$4$ & $50$ &  & $6$ & $30$ & 5.25e-11 &  & $767$ & 9.80e-11 &  & $13$ &
2.61e-14 \\
$4$ & $100$ &  & $8$ & $42$ & 6.91e-11 &  & $2000$ & 2.57e-08 &  & $92$ &
1.33e-15 \\ \hline
$5$ & $20$ &  & $4$ & $19$ & 5.99e-12 &  & $91$ & 8.31e-11 &  & $8$ &
3.18e-15 \\
$5$ & $50$ &  & $5$ & $25$ & 1.49e-17 &  & $531$ & 9.64e-11 &  & $9$ &
1.04e-13 \\
$5$ & $100$ &  & $9$ & $45$ & 6.66e-16 &  & $1918$ & 9.89e-11 &  & $80$ &
3.41e-13 \\ \hline
\end{tabular}%
\end{table}
From Table \ref{tbl:outer1},
we see that the numbers of evaluations, \#(Eval), for continuation method are between $18$ to $56$. The
convergence of NQZ \cite{Ng-Qi-Zhou:2009,Zhang-Qi:2012} is linear and the
numbers of evaluations for NQZ are between 91 to 2000.  The convergence rate of NNI \cite%
{Liu-Guo-Lin:2016,Liu-Guo-Lin:2017} is quadratic when
it is near convergence. However, the initial monotone convergence of NNI
with positive parameters $\{\theta_k\}$ may be quite slow. In this example, the numbers of
evaluations for NNI are between $7$ and $126$.

\begin{Remark}
\label{rem5.1} There is no theory to guarantee the convergence of NNI with $%
\theta_k=1$. In this example, if we employ NNI with $\theta_k=1$ to compute
the positive H-eigenpair, it has very nice performance. The number of evaluations
for NNI with $\theta_k=1$ is at most $11$.
\end{Remark}
\end{Example}
\section{Conclusions}

We have presented homotopy continuation method for computing nonnegative Z-/H-eigenpairs  of a nonnegative tensor $\mathcal{A}$. A linear homotopy $H(\mathbf{x},\lambda,t)=\mathbf{0}$ is constructed by a target nonnegative tensor $\mathcal{A}$ and a rank-1 initial tensor $\mathcal{A}_0=\mathbf{x}_1\circ\cdots\circ\mathbf{x}_1$, where $\mathbf{x}_1\in \mathbb{R}_{>0}^n$ is generic. It is shown that $H(\mathbf{x},\lambda,t)=\mathbf{0}$ has only one  positive solution at $t=0$ and the solution curve of the linear homotopy starting from the positive solution, $(\mathbf{x}(s),\lambda(s),t(s))\in \mathbb{R}_{>0}^{n+1}\times [0,1)$ for $s\in [0,s_{max})$, is smooth and  $t(s)\rightarrow 1^-$ as $s\rightarrow s_{max}^-$. Hence, the nonnegative eigenpair can be computed by tracking the solution curve if the nonnegative solutions of $H(\mathbf{x},\lambda,1)=\mathbf{0}$ are isolated. Furthermore, we have shown that the number of positive Z-eigenpairs of an irreducible nonnegative tenor is odd and proposed an algorithm to compute odd number of  positive Z-eigenpairs.  For computing nonnegative eigenpairs,  the norm of the generic positive vector $\mathbf{x}_1$ will affect the distance of two turning points and then, affect the time of computing. How to choose a suitable norm of the  generic positive vector $\mathbf{x}_1$ remains an open problem.


\end{document}